\documentclass[bj,authoryear]{imsart}

\RequirePackage{amsthm,amsmath,amsfonts,amssymb}
\usepackage{xcolor}
\usepackage{subcaption}
\usepackage{graphicx}
\usepackage{accents}
\usepackage{mathtools}
\usepackage{bbm}
\usepackage{cleveref}
\RequirePackage[authoryear]{natbib}

\allowdisplaybreaks

\startlocaldefs
\newtheorem{theorem}{Theorem}
\newtheorem{corollary}{Corollary}
\newcommand{\norm}[1]{\left\lVert#1\right\rVert}

\newcommand{\zerop}{\mathbf{0}_p}

\newcommand{\maximizer}{\hat{\theta}_n}

\newtheorem{assumption}{Assumption}
\newtheorem{remark}{Remark}[section]
\newtheorem{lemma}{Lemma}
\newtheorem{proposition}{Proposition}

\newtheorem{innercustomgeneric}{\customgenericname}
\providecommand{\customgenericname}{}
\newcommand{\newcustomtheorem}[2]{%
  \newenvironment{#1}[1]
  {%
   \renewcommand\customgenericname{#2}%
   \renewcommand\theinnercustomgeneric{##1}%
   \innercustomgeneric
  }
  {\endinnercustomgeneric}
}

\newcustomtheorem{customthm}{Theorem}
\newcustomtheorem{customlemma}{Lemma}
\newcustomtheorem{customprop}{Proof of Proposition}
\newcustomtheorem{customcor}{Corollary}
\newcustomtheorem{customass}{Assumption}

\def\[#1\]{ \begin{align} #1 \end{align} }
\def\*[#1\]{\begin{align*}#1\end{align*}}

\endlocaldefs

\begin{document}

\begin{frontmatter}
\title{Monte Carlo and quasi-Monte Carlo integration for likelihood functions}
\runtitle{MC and QMC integration for likelihood functions}

\begin{aug}
\author[A]{\fnms{Yanbo} \snm{Tang}\thanks{Corresponding author. yanbo.tang@imperial.ac.uk}}

\address[A]{Department of Mathematics, Imperial College London, London, UK}
\end{aug}

\begin{abstract}
 We compare the integration error of Monte Carlo (MC) and quasi-Monte Carlo (QMC) methods for approximating the normalizing constant of posterior distributions and certain marginal likelihoods. In doing so, we characterize the dependency of the relative and absolute integration errors on the number of data points ($n$), the number of grid points ($m$) and the dimension of the integral ($p$). We find that if the dimension of the integral remains fixed as $n$ and $m$ tend to infinity, the scaling rate of the relative error of MC integration includes an additional $n^{1/2}\log(n)^{p/2}$  data-dependent factor, while for QMC this factor is $\log(n)^{p/2}$. In this scenario, QMC will outperform MC if $\log(m)^{p - 1/2}/\sqrt{mn\log(n)} < 1$, which differs from the usual result that QMC will outperform MC if $\log(m)^p/m^{1/2} < 1$.The accuracies of MC and QMC methods are also examined in the high-dimensional setting as $p \rightarrow \infty$, where MC gives more optimistic results as the scaling in dimension is slower than that of QMC when the Halton sequence is used to construct the low discrepancy grid; however both methods display poor dimensional scaling as expected. An additional contribution of this work is a bound on the high-dimensional scaling of the star discrepancy for the Halton sequence.
\end{abstract}


\end{frontmatter}

\section{Introduction}
The error analysis for most numerical integration techniques is performed on a fixed function $f(\theta)$ as number of functional evaluations $m \rightarrow \infty$.
However, in most statistical applications the function being integrated changes as more samples (denoted by $n$) are collected.
Assuming the parameter space $\Theta \subset \mathbb{R}^p$ is compact, the normalizing constant of a posterior distribution can approximated by
\[
\int_{\Theta} \pi(\theta) \exp( L_n(\theta; Y_n) )d\theta \approx \frac{1}{m} \sum_{i = 1}^m \pi(\theta_i) \exp( L_n(\theta_i; Y_n) ),
\]
where $\pi(\theta)$ is the prior, $l_n(\theta; Y_n)$ is the log-likelihood function and $\theta_i$ for $i = 1,\dots, m$ is a collection of random or deterministic grid points in $\Theta$; 
Monte Carlo (MC) integration generates these points randomly from an uniform measure, while quasi-Monte Carlo methods (QMC) use a deterministic low discrepancy grid which does not ``clump'' together too much, see Figure \ref{fig:test} for an illustration of this. 
Beyond approximating the normalizing constant, these integration methods can also be used in random effects models to approximate the marginal likelihood
\[
L_n(\theta; Y_n) = \int_{U} \pi(u) \exp( L_n(\theta, u; Y_n) ) du \approx \frac{1}{m} \sum_{i = 1}^m \pi(u_i) \exp( L_n(\theta, u_i; Y_n) ),
\]
which can then be used to compute the marginal maximum likelihood estimator (MMLE) and generate confidence regions by estimating the curvature around the MMLE -- although to do so would require multiple evaluations of the marginal likelihood and therefore require multiple instances of MC.
Given that limit theorems (in $n$) are used to characterize the distributional properties of the MMLE, it is then natural to analyze the approximation errors of these integrals as $n \rightarrow \infty$; intuitively a reasonable condition for the approximate MMLE to be consistent is to require that the relative approximation error tends to $0$ as $n \rightarrow \infty$ uniformly around the true or pseudo-true parameter, otherwise a non-negligible bias will persist in the limit.

The existing convergence rates for MC and QMC do not account for the effect of $n$ or $p$, the most commonly quoted rate of convergence of Monte Carlo integration being $O(m^{-1/2})$ and $O(\log(m)^p m^{-1})$ for QMC.
The benefit of working with random functions is that they concentrate towards a population version which is better behaved than an arbitrary function. 
We exploit these structures to derive bounds for MC or QMC where the dimensional dependencies of the integration errors are explicit, and compare their performances under reasonable smoothness assumptions. 
These upper bounds in fixed and high dimensions will be useful in providing a baseline comparison for Markov Chain Monte Carlo (MCMC) and sequential Monte Carlo (SMC), which also suffer the curse of dimensionality. 
For a particular MCMC method to be considered viable, it must at the very least perform better than simple MC; although bounds for the performance of MCMC are much harder to obtain than for MC and QMC.

We first give some background on MC and QMC in Section \ref{sec:background}, before moving on to our main results for integrating the normalizing constant of a posterior distribution for the fixed-dimensional and high-dimensional case in Section \ref{sec:mainresult}. We then examine the implications of our results for marginal likelihoods in Section \ref{sec:mixed_models}, and provide some simulation results in Section \ref{sec:sims} before concluding the main portion of the paper with some discussions in Section \ref{sec:discussion}. Technical appendices \ref{app:MC} and \ref{app:QMC} contain proofs of lemmas required for our main results for MC and QMC integration respectively.

\section{Background}\label{sec:background}
MC integration leverages the randomness of the grid points by using the strong law of large numbers and the central limit theorem (CLT) to show consistency and rates of convergence of the estimate under first and second moment assumptions.
However these results require that the function remains fixed: if both $m$ and $n$ change, then it is no longer clear if CLT applies to this sequence. 
We take an approach similar to that of \cite{TangMC}, where concentration inequalities are used instead of the CLT. The advantage of this approach is that it is possible to derive results when the dimension of the integral increases with the number of samples, which is necessary for the asymptotic normality of the MMLE in generalized linear mixture models, for example see \cite{jiang2022usable} and \cite{jiang2024preciseasymptoticslinearmixed}.

The integration error for QMC is obtained by bounding the star discrepancy of the grid points and the Hardy-Kraus (HK) variation of the function being integrated, and then applying Koksma-Hlawka's inequality; let us define these terms.
Define the local discrepancy of a sequence $(x_1, \dots, x_m)$--for $x_i \in [0,1]^p$--as
\*[
    \delta(a; x_1, \dots, x_m) = \frac{1}{m} \sum_{i = 1}^m \mathbb{I}_{x_i \in [0, \bold{a}) } - \prod_{j = 1}^p a_j,  
\]
where $\bold{a} \in [0,1]^p$ and $[0, \bold{a}) = [0, a_1) \times [0, a_2) \times \dots \times [0, a_p)$ and we can think of the discrepancy between the estimated volume of rectangular sets based on our sample $(x_1, \dots, x_m)$ and their true volume.
We then define the star discrepancy as
\[
    D_m^\star(x_1, \dots, x_m) = \sup_{\bold{a} \in [0,1)^p} \delta(\bold{a}; x_1, \dots, x_m).
\]
The class of functions considered in this paper is assumed to be $p$ times differentiable, and this assumption implies an upper bound of the HK variation involving the $p$-th order partial derivatives of the function $f(x)$ on the domain of integration:
\*[
V_{HK}(f) \leq \sum_{\alpha \subset \mathcal{P}\{1,2,\dots, p\}-  \emptyset} \int_{[0,1]^p} \frac{\partial^{|\alpha|}}{\partial x_{\alpha}} f(x_\alpha, 1_{-\alpha}) dx_{\alpha},
\]
where $\mathcal{P}(A)$ denotes the power set of $A$, and the summation is taken over all subsets of $1:p$ except the empty set, $x_\alpha$ denotes the subset of the vector $x$ whose indices are in $\alpha$ and $1_{-\alpha}$ denotes a set of $1$ on all of the indices which are not in $\alpha$.
For some intuition, it is worth noting that the HK variation in one dimension is the total variation of the function in $[0,1]$.

The Koksma-Hlawka inequality then upper bounds the absolute error between the integral and its estimate by product of the two quantities introduced above
\[\label{eq:KH}
\left| \int_{[0,1]^p} f(x) dx - \frac{\sum_{i = 1}^m f(x_i)}{n} \right| \leq D_m^\star (x_1, \dots, x_m)\cdot V_{HK}(f),
\]
where this inequality holds for any sequence of grid points $x_1, \dots, x_m$, deterministic or random. 

For a fixed function, the HK variation of the function is fixed and usually assumed to be finite. Therefore, minimizing the upper bound in \eqref{eq:KH} involves selecting sequences with the lowest star discrepancy possible.
Standard low discrepancy grids will result in $D_m^\star(x_1, \dots, x_m) = O(\log(m)^{p}/m)$ or $O(\log(m)^{p -1}/m)$, which scales rapidly with $p$ on the numerator, but also decays more rapidly on the denominator by a factor of $\sqrt{m}$ compared to the MC rate obtained through the CLT. 
In this work, we consider the Halton sequence as a running example of a low discrepancy sequence to further our understanding of QMC as $n$ and potentially $p$ increase.
In order to define the Halton sequence, consider the digit retrieval function $d_{k,b}(i) \in \{0,1,\dots, b-1 \}$ which are digits in the base $b$ expansion of an integer $i$, i.e. $i = \sum_{k = 1}^\infty d_{k,b}b^k$, and define $\phi_{b}(i) = \sum_{k = 0}^\infty d_{k,b}(i)$.

The Halton sequence $(x_1, \dots, x_m)$ is then generated component-wise by
\[
x_{ij} = \phi_{c_j}(i - 1), 
\]
for $i \leq m$ and $1\leq j \leq p$ where $c_j$ is the $j$-th prime number; although it is sufficient for $c_j$ to be a sequence of co-primes, see \cite{atanassov2004discrepancy}.

\begin{figure}
    \centering
    \begin{subfigure}{.5\textwidth}
      \centering
      \includegraphics[width=.95\linewidth]{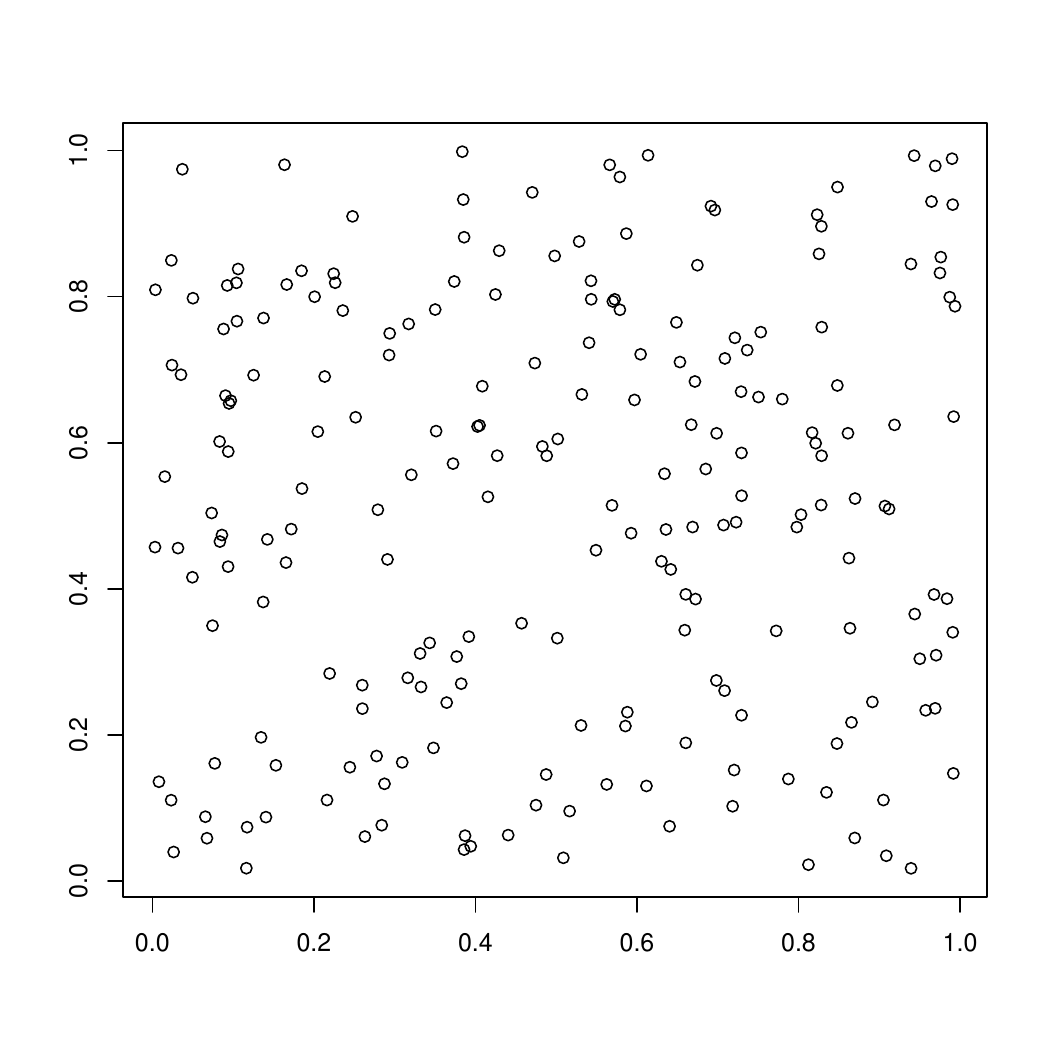}
      \caption{Random uniform sequence}
      \label{fig:sub1}
    \end{subfigure}%
    \begin{subfigure}{.5\textwidth}
      \centering
      \includegraphics[width=.95\linewidth]{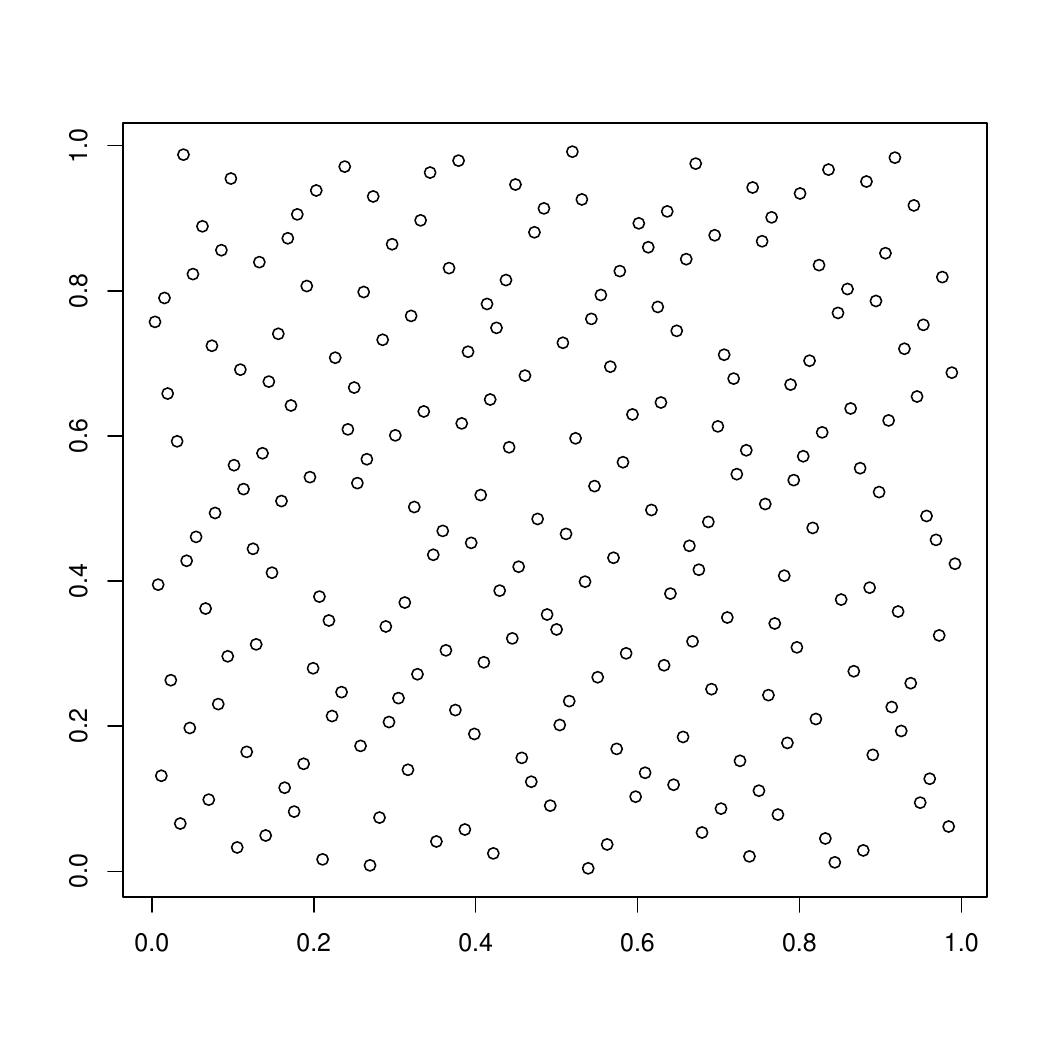}
      \caption{Halton sequence}
      \label{fig:sub2}
    \end{subfigure}
    \caption{A comparison of a random uniform sample on the left and the Halton sequence on the right for m = 216. Note that the uniform samples tend to cluster closely together in several locations while the low discrepancy sequence does not.}
    \label{fig:test}
    \end{figure}

Complications arise when $n$ and/or $p$ is allowed to increase for the analysis of QMC.
As the size of the derivatives of the log-likelihood function can potentially increase with $n$, directly upper bounding the $VH$ variation of the likelihood function can result in a exponentially growing upper bound in $n$. 
To limit this potential growth, we truncated the integral to a shrinking region around the maxima of the function ($\hat\theta_n$) and show that the truncation error is of smaller order than the integration error.

Let $\mathscr{L}_n(\theta)$ denote the product of the likelihood function and the prior density.
We integrate this function over a shrinking hypercube $[\hat\theta_n - \gamma_n, \hat\theta_n + \gamma_n]^p$, therefore we need to instead bound the HK variation of the rescaled function $\mathscr{L}_n(2\gamma_n\cdot (x- 1/2) + \hat\theta_n )$ for $x \in [0,1]^p$, a sequence of $\gamma_n \rightarrow 0$ to be specified later and where $\hat\theta_n$ is the maxima of the function $\mathscr{L}_n(\theta)$. Specifically,
\*[
V_{HK}\{\mathscr{L}_n(2\gamma_n\cdot (x- 1/2) + \hat\theta_n) \} = (2\gamma_n)^p \sum_{\alpha \subset \mathbb{P}\{1,2,\dots, p\}-  \emptyset} \int_{[0,1]^p} \frac{\partial^{|\alpha|}}{\partial_{\alpha}}\mathscr{L}_n(2\gamma_n\cdot (x- 1/2)+ \hat\theta_n) dx,
\]
where the $(2\gamma_n)^p$ factor comes from the change of variable from rescaling the grid $[\hat\theta_n - \gamma_n, \hat\theta_n + \gamma_n]^p$ to $[0,1]^p$;
as $\gamma_n$ tends to $0$ this multiplicative factor will reduce the size of these partial derivatives and therefore improve the upper bound on the HK variation.

In the sequel, we use $A(n,m, p)$ and $R(n,m, p)$ to denote the absolute and relative approximation errors on the truncated domain respectively, and we use $E(n,m,p)$ and $RE(n,m,p)$ to denote the overall absolute and relative approximation errors of MC and QMC.
The relative error is of greater importance for most applications, as the ratio of the normalizing constants or of the marginal likelihood is needed for inference rather than their difference.
In fact, given that the normalizing constant and marginal likelihood shrink exponentially quickly as more data is obtained, it is very easy to obtain an estimate that is close in absolute value.

For more background details on the construction of these low discrepancy grids and more generally on other methods of QMC, see  \cite{carlo2009quasi} or \cite{owen2019monte}.

\section{Main results}\label{sec:mainresult}
\subsection{Assumptions}

Let $p$ denote the dimension of the integral, $n$ the number of samples and $m$ the number of grid points.
Let $\lambda_i(A)$ denote the $i$-th largest eigenvalue of a semi-positive definite symmetric matrix $A$.
Define $B_a(b)$ as the $L^2$ ball centered at $a$ with radius $b$, and define $B^\infty_a(b)$ as a $L^\infty$ ball (or a hypercube) centered at $a$ with radius $b$.
Let $\gamma_n^2 = \eta_1 t(n)/n$ and let $(\gamma^\prime_n)^2 = \eta_1 p t(n)/n$, for $\eta_1$ defined in Assumption \ref{assn:hess} and for some function $t(n)$ which tends to $\infty$ as $n\rightarrow \infty$.
We use $L_n(\theta; X_n)$ to denote the log-likelihood function, and $l_n(\theta; X_n) = L_n(\theta; X_n) + \log(\pi(\theta))$ which is the log-likelihood combined with the log-prior.
Let $\delta_{n,p} > \sqrt{p}\gamma_n^\prime >0$.

\begin{assumption} \label{assn:delta_decay}
There exists an $\epsilon > 0$ such that
\begin{align*}
    & \limsup_{\theta \in B^C_{\maximizer}(\delta_{n,p})} L_n(\theta; X_n)  - L_n(\maximizer ; X_n) \leq -n^{\epsilon},
\end{align*}  
with probability tending to $1$ as $ n\rightarrow \infty$.
\end{assumption}

\begin{assumption} \label{assn:hess}
The eigenvalues of the Hessian of $l_n(\theta; X_n)$  satisfy:
\*[ 0< \eta_1 n \leq \lambda_p[ -l^{(2)}_n(\theta) ] \leq \lambda_1[ -l^{(2)}_n(\theta)] \leq \eta_2 n < \infty ,\]
for some $\eta_1,\eta_2 >0$ and for all $\theta \in B_{\maximizer}(\delta_{n,p})$ with probability $1 - h(\delta_{n,p},n,p)$ for some function $h(\delta_{n,p},n,p)$ such that $\lim_{n\rightarrow \infty} h(\delta_{n,p}, n,p) = 0$.
\end{assumption}

\begin{assumption}\label{assn:partial}
    The partial derivatives of $l_n(\theta; X_n)$ satisfy:
    \*[  \left|\frac{\partial^{k}}{\partial\theta_{i_1} \cdots \partial\theta_{i_k}}l_n(\theta; X_n) \right| \leq D n < \infty ,\]
for all $\theta \in B^\infty_{\maximizer}(\gamma_n^\prime)$, for all integers $k \leq p$, and $i_1 +\dots +i_k = k$ with probability tending to $1$ as $ n\rightarrow \infty$.
\end{assumption}

\begin{assumption}\label{assn:prior}
    The prior distribution $\pi(\theta)$ has a density with respect to the Lebesgue measure and
    \*[
    \int_{\Theta} \pi(\theta) d\theta < \infty.
    \]
\end{assumption}

\begin{remark}
    Note that Assumption \ref{assn:partial} also requires the prior to be $p$ times differentiable as we have combined the prior along with the likelihood function, this rules out the use of non-smooth prior such as the Laplace distribution, for example.
\end{remark}

The restriction that $\delta_{n,p} > \sqrt{p}\gamma_n^\prime$ is to guarantee that $B_{\hat\theta_n}(\gamma_n^\prime) \subset B_{\hat\theta_n}(\delta_{n,p})$, which is required for some steps of the proofs to follow.
Only Assumptions \ref{assn:delta_decay} and \ref{assn:hess} are required for the bounds on the MC integration error using a random uniform grid. Assumption \ref{assn:delta_decay} is a weaker version of the assumptions typically used to show the consistency of the MLE; where $\epsilon = 1$ is used. Assumption \ref{assn:hess} is very commonly used and controls the local curvature around the mode.
Assumption \ref{assn:partial} is needed to control the HK variation of the functions locally which is required for our bound on the QMC error.
The rate at which the probabilities in Assumptions \ref{assn:delta_decay}--\ref{assn:partial} tend to $1$ can be obtained through the usual applications of empirical process theory.

In the sections that follow, we analyze the following procedure in which we truncate the region of integration to a square centered around the posterior mode. Specifically, we analyze the relative and absolute accuracies of the following estimator for the normalizing constant
\*[
 \sum_{i = 1}^m \frac{ (2\gamma_n)^p }{m} \exp\left\{l_n(2\gamma_n\cdot (x_i- 1/2)+ \hat\theta_n) \right\},
\]
where $\hat\theta_n$ is the mode of the posterior distribution.
This procedure requires the additional step of maximizing the posterior density, but truncating the region of integration this way ensures that less points are wasted on regions of low mass.

\subsection{Dimension $p$ fixed}
The bounds subsequently derived are valid for all values of $m$, $n$ and $p$, and any omitted constants in $O()$ statements are independent of $n$ and $m$, although there may be some dependency on $p$. 
We first consider MC integration. We truncate the integral to a $L^\infty$ ball of radius $\gamma_n^2 = t(n)/(\eta_2 n)$; this radius can be interpreted as an inflated rate of posterior contraction of $n^{-1/2}$ for parametric models. 
We control the truncation error with the lemma which follows.

\begin{lemma}\label{lemma:truncation}
     Define $\gamma_n^2 = t(n)/(\eta_2 n)$ for a function $t(n)\rightarrow \infty$, under Assumptions \ref{assn:delta_decay}, \ref{assn:hess} and \ref{assn:prior} there exists an $N_0$ such that for all $n \geq N_0$
     \*[
        &\int_{ \{B_{\hat\theta_n}(\gamma_n)\}^C } \pi(\theta) \exp( l_n(\theta; Y_n) - L_n(\hat\theta_n; Y_n))d\theta = O\left\{ \frac{1}{n^{p/2}} \exp( -\min\left\{ n^\epsilon, pt(n)/2 \right\} ) \right\},
        \]
    and,
     \*[
    &\frac{\int_{ \{B_{\hat\theta_n}(\gamma_n)\}^C } \pi(\theta) \exp( L_n(\theta; Y_n) - L_n(\hat\theta_n; Y_n))d\theta}{\int_{ \mathbb{R}^p } \pi(\theta) \exp( L_n(\theta; Y_n) - L_n(\hat\theta_n; Y_n))d\theta}  = O\left\{  \exp( -\min\left\{ n^\epsilon, pt(n)/2 \right\} ) \right\},
    \]
    with probability tending to 1 for $\epsilon$ as defined in Assumption \ref{assn:delta_decay}.
\end{lemma}

\begin{proof}
    By Assumption \ref{assn:delta_decay} there exists an $N_0$ such that for all $n > N_0$

    \begin{align*}
        \int_{ \{B_{\hat\theta_n}(\delta_{n,p})\}^C } \pi(\theta) \exp( L_n(\theta; Y_n) - L_n(\hat\theta_n; Y_n))d\theta &\leq 
        \int_{ \{B_{\hat\theta_n}(\delta_{n,p})\}^C } \pi(\theta) \exp\left( -\frac{3}{4}n^\epsilon \right) d\theta \\
        &\leq \exp\left( -\frac{3}{4}n^\epsilon \right),
    \end{align*}
    as $\pi(\theta)$ is a density by Assumption \ref{assn:prior}. From the proof of Lemma \ref{lemma:A-trunc} (Equation \ref{eq:lower_bound}) we also have the intermediary result of
    \begin{align*}
        &\int_{ \mathbb{R}^d } \pi(\theta)  \exp\{ L_n(\theta; X_n)  - L_n(\maximizer ; X_n)  \} d\theta \geq \frac{1}{2} \left(\frac{2\pi\eta_2}{n} \right)^{p/2}, 
    \end{align*}
    combining these two bounds and Lemma \ref{lemma:A-trunc} gives the stated result.
\end{proof}

\begin{remark}
    If Assumption \ref{assn:delta_decay} is changed such that there exists an $\epsilon > 0$ and a $C>0$ such that for all $n$
    \begin{align*}
        & \sup_{\theta \in B^C_{\maximizer}(\delta_{n,p})}  L_n(\theta; X_n)  - L_n(\maximizer ; X_n) \leq -Cn^{\epsilon},
    \end{align*} 
    then $N_0 = 1$ in Lemma \ref{lemma:truncation}. 
\end{remark}

As a reminder, $A(n,m,p), R(n,m,p)$ denote the absolute and relative integration errors on the truncated domain.
\begin{theorem}[MC]
    Under Assumptions \ref{assn:delta_decay}, \ref{assn:hess} and \ref{assn:prior} when using a uniform random grid with $\gamma_n$ and $t(n)^2 = \eta_2\log(n)/\eta_1$
     \*[
        P[ A(n,m,p) >  \zeta] \leq 2\exp\left( \frac{-1}{4} \frac{\eta_1 \eta_2^{p -2}  m n^{p-1} \zeta^2}{ 4^p t(n)^{p+1} p} \right)+ h(\delta_{n,p}, n, p), 
    \]
    and
    \*[
    P \left\{ R(n,m,p) > \zeta \right\} &\leq 2\exp\left\{ \frac{-1}{4} \left(\frac{2\pi}{4} \right)^p \frac{\eta_1 \eta_2^{2p -2}  m  \zeta^2}{ t(n)^{p+1} n p} \right\} + h(\delta_{n,p} , n, p).
    \]
\end{theorem}

\begin{proof}
    Let $X_m = (x_1, \dots, x_m)$ where $x_i \sim Unif[0,1]^p$ define
    \*[
        f(X) := \frac{(2\gamma_n)^p}{m}\sum_{i = 1}^m \exp\{ l_n(2\gamma_n(x_i -1/2); Y_n) - l_n(\hat\theta_n; Y_n)\},
    \]
     by Assumption \ref{assn:hess} $f(X_m)$ is a strongly convex function with probability $1 - h(\delta_{n,p}, n,p)$, from Lemma \ref{lem:lip} the Lipschitz constant of the function 
     \*[
        \exp\{ l_n(2\gamma_n(x_i -1/2); Y_n) - l_n(\hat\theta_n; Y_n)\}
     \]
    is $ \{pnt(n)\}^{1/2}\eta_2/\eta_1^{1/2}$ on $[0,1]^p$, which implies the Lipschitz constant of the function $f(X_m)$ is 
    $$\frac{2^p p^{1/2} t(n)^{\frac{p+1}{2}}}{\{\eta_1 \eta_2^{p - 2} m n^{p-1}\}^{1/2}}.$$
    Using Theorem 3.24 in \cite{Wainwright_2019}, we then have the following tail bound
    \*[
        P[ \left|f(X_m) - E\{f(X_m)\}\right| \geq \zeta ] \leq 2\exp\left( \frac{-\eta_1 \eta_2^{p -2}  m n^{p-1} \zeta^2}{ 4^p t(n)^{p+1} p} \right), 
    \]
    which then gives
    \*[
        P[ \left|f(X_m) - E\{f(X_m)\}\right| \geq \epsilon E\{f(X_m)\} ] \leq 2\exp\left\{ \frac{-1}{4} \left(\frac{2\pi}{4} \right)^p \frac{\eta_1 \eta_2^{2p -2}  m  \zeta^2}{ t(n)^{p+1} n p} \right\}, 
    \]
    by the lower bound on the normalizing constant from \eqref{eq:lower_bound}, and applying the union bound on the event that the function is possibly not convex then gives the desired result.
\end{proof}

These tail bounds can be reformulated into a more familiar error rate by choosing a specific $t(n)$ and an $\epsilon$ which is a function of $n$ and $m$. As a reminder, $AE(n,m,p), RE(n,m,p)$ denote the absolute and relative integration errors of the entire approximation.

\begin{corollary}[MC]
    Under Assumptions \ref{assn:delta_decay}, \ref{assn:hess} and \ref{assn:prior} when using an uniform random grid with $\gamma_n$ with $t(n)^2 = \eta_2\log(n)/\eta_1$
    \*[
        E(n,m,p) \leq 2^p \sqrt{p} \left(\frac{ \eta_2 }{\eta_1} \right)^{3/2} \left( \frac{\log(n)^{p+1} \log(m)}{ n^{p-1} m} \right)^{1/2} + O\left\{ \frac{1}{n^{p}}  \right\}
        = O\left\{ \frac{\log(n)^{p+1}\log(m)}{n^{p-1}m}\right\}^{1/2}
    \] 
    and 
    \*[
        RE(n,m,p) &\leq 2 \sqrt{p} \left(\frac{4}{2\pi}\right)^{p/2} \left( \frac{  \eta_2 }{\eta_1}\right)^{3/2} \left( \frac{n\log(n)^{p+1} \log(m)}{ m } \right)^{1/2} + O\left\{ \frac{1}{n^{p/2}} \right\} \\
        &= O\left\{ \frac{n\log(n)^{p+1}\log(m)}{m}\right\}^{1/2}
    \] 
    with probability $1 - 2/m - h(\delta_{n,p},n)$.
\end{corollary}

\begin{proof}
    This follows by noting that
    \*[
    P\left[ \left|f(X_m) - E\{f(X_m)\}\right| \geq 2^p \sqrt{p} \left(\frac{ \eta_2 }{\eta_1} \right)^{3/2} \left( \frac{\log(n)^{p+1} \log(m)}{ n^{p-1} m} \right)^{1/2}\right] \leq 2\exp\left( \log(m) \right) = \frac{2}{m},
    \]
    adding the truncation error from Lemma \ref{lemma:truncation} then gives the result; the bound for the relative error is similar.
\end{proof}

The absolute integration error decays rapidly with $n$ as the normalizing constant naturally becomes smaller as dimension increases, while the relative error increases at a rate of $n\log(n)^{(p+1)/2}$ as dimension increases, which mirrors the polylogarithmic dependency which appears in QMC.

We now discuss the integration error of QMC, and in what follows we will assume that the sequence used to integrate the function has a discrepancy of order $O(\log(m)^p/m)$.
The radius of the truncation region for QMC is larger than that of MC at $(\gamma_n^\prime)^2 = \eta_2p\log(n)/n$ by an order of $p^{1/2}$, and this extra factor comes from bounding the HK variation of this sequence of function.

\begin{theorem}[QMC]
    Under Assumptions \ref{assn:delta_decay}--\ref{assn:prior}, using a sequence with star discrepancy of order $O(\log(m)^p/m)$ and $\gamma_n^\prime$ with  $t(n) = \eta_2\log(n)/\eta_1$ for $p \geq 3$
    \*[
    EA(n,m,p) = O\left(\frac{\log(m)^p}{m} \frac{\log(n)^{p/2}}{n^{p/2}}\right),
    \]
    and
    \*[
    RE(n,m,p) = O\left(\frac{\log(m)^p \log(n)^{p/2}}{m }\right) .
    \]
\end{theorem}

\begin{proof}
    We use the Koksma-Hlawka inequality \eqref{eq:KH} to bound the integration error. As the star discrepancy is already known, it remains to bound the HK variation of the truncated likelihood function which is $O\{(\gamma_n^\prime)^p\}$ by Lemma \ref{lem:hk-var} and noting that the truncation error is of smaller order by Lemma \ref{lemma:truncation}.
\end{proof}
Notably the absolute error for QMC is less dependent on the sample size $n$ by a factor of $n$ versus MC, although the problematic $\log(n)^{p/2}$ factor persists for relative accuracy.

\begin{remark}
Although the HK variation is of $O\{(\gamma^\prime_n)^p\}$ in $n$, there are large combinatorial constants involved in the bound, so practically it can still be quite large in the worst case scenario.
We track these constants as $p \rightarrow \infty$ in the next subsection, we then provide an example in which these constants are smaller.
\end{remark}

When comparing the integration error on the truncated region for MC and QMC, as the truncation error is exponentially tending to $0$,
QMC performs better than MC in terms of the rate of convergence if 
\*[
\frac{\log(m)^{p - 1/2}}{\sqrt{m\log(n)n}} < 1,
\]
which is more favourable than the analysis for a fixed function, where QMC performs better than MC if
\*[
\frac{\log(m)^{p}}{\sqrt{m}} < 1,
\]
particularly if $n$ is very large.

\subsection{Dimension $p$ growing with $n$}
We now examine the performance of these methods should dimension of the integral be allowed to increase.
To motivate this regime, we note that for certain types of mixed models, asymptotic guarantees require that the dimensionality of the integral increases as the number of samples increases.  
The derived bounds allow for this dependency as well, but as will be shown, the error rates will be worse and scale with $p$ exponentially.
This worsening of the rate can be seen as being caused by the curse of dimensionality, where the region of integration grows exponentially in terms of volume as we increase $p$. 
\begin{corollary}
    Under Assumptions \ref{assn:delta_decay}, \ref{assn:hess} and \ref{assn:prior} using $\gamma^{\prime}_n$ with $t(n)^2 = \eta_2\log(n)/\eta_1$ to truncate the integral 
    \*[
    \int_{ \{B_{\hat\theta_n}(\gamma^{\prime}_n)\}^C } \pi(\theta) \exp( L_n(\theta; Y_n) - L_n(\hat\theta_n; Y_n))d\theta = O\left\{ n^{-p + \epsilon} \right\},
    \]
    with probability $1 - h(\delta_{n,p},n,p)$ and
    \*[
    \frac{\int_{ \{B_{\hat\theta_n}(\gamma^{\prime}_n)\}^C } \pi(\theta) \exp( L_n(\theta; Y_n) - L_n(\hat\theta_n; Y_n))d\theta}{\int_{ \mathbb{R}^p } \pi(\theta) \exp( L_n(\theta; Y_n) - L_n(\hat\theta_n; Y_n))d\theta} = O\left(n^{-p/2} \right).
    \]
\end{corollary}
This corollary is a consequence of Lemma \ref{lemma:truncation} for our choice of $t(n)$.
The proof of the theorem to follow remains the same as the fixed-dimensional case and will not be repeated.

\begin{theorem}[MC]
    Under Assumptions \ref{assn:delta_decay}, \ref{assn:hess} and \ref{assn:prior}, if $p$ is allowed to increase with $n$ and using $\gamma^{\prime}_n$ with $t(n)^2 = \eta_2\log(n)/\eta_1$ to truncate the integral with an uniformly random grid
    \*[
    P \left\{ A(n,m,p) > \zeta \right\} &\leq 2\exp\left(-\frac{1}{4^p} \frac{\eta_1 \eta_2^{p -2}  m n^{p-1} \zeta^2}{  t(n)^{p+1} p^{p+2}} \right) + h(\delta_{n,p},n,p),
    \]
    and
    \*[
    P \left\{ R(n,m,p) > \zeta \right\} &\leq 2\exp\left\{- \frac{1}{4} \left(\frac{2\pi}{4} \right)^p \frac{\eta_1 \eta_2^{2p -2}  m  \zeta^2}{ t(n)^{p+1} n p^{p+2}} \right\} + h(\delta_{n,p},n,p).
    \]
\end{theorem}
And we obtain the result that:

\begin{corollary}[MC]
    Under Assumptions \ref{assn:delta_decay}, \ref{assn:hess} and \ref{assn:prior}, if $p$ is allowed to increase with $n$ and using $\gamma^{\prime}_n$ with $t(n)^2 = \eta_2\log(n)/\eta_1$ to truncate the integral with an uniformly random grid
    \*[
        E(n,m,p) &\leq 2^p p^{p/2 + 1} \left(\frac{ \eta_2 }{\eta_1} \right)^{3/2} \left( \frac{\log(n)^{p+1} \log(m)}{ n^{p-1} m} \right)^{1/2} + O\left\{ \frac{1}{n^{p}}  \right\}\\
        &= O\left( C^p \sqrt{\frac{p^{p +2} \log(n)^{p+1} \log(m)}{n^{p-1}m}}\right)
    \] 
    and 
    \*[
        RE(n,m,p) &\leq 2 p^{p/2 + 1} \left(\frac{4}{2\pi}\right)^{p/2} \left( \frac{  \eta_2 }{\eta_1}\right)^{3/2} \left( \frac{n\log(n)^{p+1} \log(m)}{ m } \right)^{1/2} + O\left\{ \frac{1}{n^{p/2}} \right\} \\
        &=  O\left( C^p \sqrt{\frac{p^{p +2} n\log(n)^{p+1} \log(m)}{m}} \right)
    \] 
    with probability $1 - 2/m - h(\delta_{n,p},n,p)$ for some $C >0$.
\end{corollary}

Meaning that for MC, if $m \sim p^{p} \log(n)^p n$ then the relative approximation error tends to $0$, which is exponentially increasing in dimensions as one might expect.
We now consider the integration error of QMC in high dimensions. A more thorough analysis of the star discrepancy is needed here as the scaling of the discrepancy for commonly used sequences is not usually shown with the dimension of the integral increasing; we perform the analysis here for the Halton sequence, although similar bounds can be derived for other sequences.

\begin{lemma}[QMC]\label{lem:halton}
    The Halton sequence satisfies the following:
    \*[
        D^\star(x_1, \dots, x_m) = O\left(  \frac{(4e)^p p^{3/2} \log(p) \log(m)^p}{ m}\right),
    \]
    in the joint limit as both $m$ and $p$ tend to $\infty$.
\end{lemma}

Specifically, this result shows that the constant which scales with $p$ is of order $O\{ (4e)^p p^{3/2}\log(p) \}$.

\begin{proof}
    From Theorem 2.1 of \cite{atanassov2004discrepancy} we have the following upper bound if $c_1, \dots, c_p$ are all distinct prime numbers
    \*[
    m D^\star(x_1, \dots, x_m) \leq \frac{2^p}{p!} \prod_{i = 1}^p \left( \frac{(c_i - 1)\log(m)}{2\log(c_i)} + p \right) + 2^p \left(c_1 + \sum_{k = 1}^{p-1}\frac{c_{k+1}}{k!}\prod_{i = 1}^k\left( \left\lfloor \frac{c_i}{2} \right\rfloor \frac{\log(m)}{\log(c_i)} +k \right) \right).
    \]
    We bound the two terms separately. From \cite{rosser1941explicit} we have the following upper bound for all $n \geq 6$
    \*[
    c_n \leq n\log(n) + n\log\log(n)
    \]
    which implies the much cruder bound:
    \*[
    c_n \leq n\log(n) + n\log\log(n) + 2
    \]
    for all $n \geq 1$ as $\max_{i = 1, \dots, 5} |c_i - i\log(i) - i\log\log(i)| < 2$. Similarly we have a lower bound on $c_n$ from \cite{rosser1941explicit}
    \*[
        c_n \geq n\log(n)
    \]
    which is valid for all $n$.
    Substituting this into the first term, we have:
    \*[
        \frac{2^p}{p!} \prod_{i = 1}^p \left( \frac{(c_i - 1)\log(m)}{2\log(c_i)} + p \right) 
        &\leq \frac{2^p}{p!} \prod_{i = 1}^p \left( \frac{(i\log(i) + i\log\log(i)  + 1)\log(m)}{2\log(i\log(i))} + p \right)\\
        &\leq \frac{2^p}{p!} \left[ \sum_{i = 1}^p \left\{ \left( \frac{i}{p} + \frac{1}{2p} \right)\log(m) + 1\right\} \right]^p\\
        &\leq \frac{2^p}{p!} \left\{ (p+3/2)\log(m) + p \right\}^p 
    \]
by the arithmetic mean-geometric mean (AM-GM) inequality. Applying the lower bound of Stirling's approximation, we have
 \begin{align*}
        \frac{2^p}{p!} \left\{ (p+3/2)\log(m) + p \right\}^p &\leq \frac{2^p}{p!} \left( 1 + \frac{2}{\log(m)}\right)^p  p^p \log(m)^p \\
        &\leq \frac{e^{\frac{1}{13} }(4e)^p}{\sqrt{2\pi p}}\log(m)^p,
\end{align*}
for $m > e^2$.
For the second term, once again applying the lower bound of Stirling's approximation
    \*[&2^p \left(c_i + \sum_{k = 1}^{p-1}\frac{c_{k+1}}{k!}\prod_{i = 1}^k\left( \left\lfloor \frac{c_i}{2} \right\rfloor \frac{\log(m)}{\log(c_i)} +k \right) \right)\\
    &\leq 2^p \left(2 + \sum_{k = 1}^{p-1}\frac{(k+1)\{\log (k+1) + \log\log (k+1) +2 }{k!} \prod_{i = 1}^k\left( \frac{\log(m)\{i\log(i) + i\log\log(i) + 2\} }{\log(i\log(i))} + k\right) \right)\\
    &\leq 2^p \left(2 + \sum_{k = 1}^{p-1} \left[ (k+1)\{\log (k+1) + \log\log (k+1) +2  \right] \frac{\left\{ (k+3/2)\log(m) + k \right\}^k}{k!} \right) \\
    &\leq 2^p \left(2 + \frac{e^{p-1}\left\{\left(1  + \frac{3}{2(p-1)} \right)\log(m)\right\}^{p-1}}{\sqrt{2\pi (p - 1) }} \sum_{k = 1}^{p-1} (k+1)\{\log (k+1) + \log\log (k+1) +2 \}\right)\\
    &\leq 2^p \left(2 + \frac{3e^{p+1}p^2\log(p) \log(m)^{p-1}}{\sqrt{2\pi (p - 1) }} \right)
    \]
   as 
    \*[
        &\sum_{k = 1}^{p-1} (k+1)\{\log (k+1) + \log\log (k+1) +2\} \leq \sum_{k = 1}^{p-1}3(k+1)\log (k+1) \\
        &\leq 3p^2\log(p).
    \] 
    Combining these two bounds then implies the stated result.
\end{proof}

This discrepancy bound will allow us to give an upper bound on the performance of QMC in high dimensions when combined with our bound on the HK variation of the likelihood function in high dimensions.

\begin{corollary}\label{cor:QMC-HD}
    Using the Halton sequence to a function of dimension $p \geq 3$ satisfying Assumptions \ref{assn:delta_decay}--\ref{assn:prior} and using $\gamma^{\prime}_n$ with $t(n)^2 = \eta_2\log(n)/\eta_1$ to truncate the integral results in
    \*[
        EA(n,m,p) = O\left( C^p \frac{\log(m)^p}{m} \frac{ p^{ (3p + 5)/2 }}{\log(p+1)^{p-1} n^{p/2+1}}  \right),
    \]
    and
    \*[
        ER(n,m,p) = O\left( C^p \frac{\log(m)^p}{m}   \frac{p^{ (3p +5)/2 } \log(n)^{p/2}}{\log(p +1)^{p-1} n}   \right),
    \]
    for some $C > 0$.
\end{corollary}
The performance of QMC is notably worse than in the fixed-dimensional case, in particular with the inclusion of $p^{(3p+5)/2}$ factor which grows extremely rapidly with $p$.
Once again we can compare the efficiency of these methods, but now accounting for possible constants in $p$.
QMC will outperform MC if
\*[
    C^p \frac{\log(m)^{p -1/2}}{\sqrt{m}} \frac{ p^{p + 2}}{\log(p+1)^p\sqrt{\log(n)}n^{3/2}} < 1
\]
where the additional factor, compared to the classical result, scales extremely rapidly in dimension.
\subsection{A simple example of improved rates for QMC}\label{sec:normal}
We consider a simple example in which we simultaneously show how some of the assumptions can be verified and how the rates can be improved. 
It is possible to reduce the scaling of the HK variation of the function by considering functions with small or vanishing derivatives of the log-likelihood.
One good example is the isotropic multivariate Gaussian, where all derivatives of the log-likelihood of order greater than $2$ vanish.

Consider the multivariate Gaussian model with density for an observation $y \in \mathbb{R}^p$
\[
f_\mu(y) =  \frac{1}{\sqrt{2\pi \sigma^p}} \exp\left( -\frac{1}{2\sigma}\lVert y -\mu\rVert^2 \right),\label{ex:gaussian}
\]
we assume $\sigma$ to be known and $\mu \subset \mathbb{R}^p$ is the unknown parameter. To simplify the calculations, further assume that the true mean $\mu_0 = (0, \dots, 0)$ and that we use a multivariate normal prior with iid components on $\mu$ with mean $0$ and covariance matrix $\sigma_p^2 I_p$ resulting in the following normalizing constant for the posterior distribution based on observations $y_1, \dots, y_n$: 
\*[
\int_{\mu \in \mathbb{R}^p} \exp\left( -\frac{1}{2\sigma}\sum_{i = 1}^n\lVert y_i -\mu\rVert^2_2 - \frac{\lVert \mu \rVert_2^2}{2\sigma_p} \right)d\mu.
\]

\begin{proposition}
   For $p/ n \rightarrow 0$ and for the model given in (\ref{ex:gaussian}) the following holds
    \*[
    RE(n,m,p) = O\left( \frac{(8e)^p \log(m)^p \log(p) p^{p/2 + 3/2}\log(n)^{p/2}}{m n}  \left(1 + \frac{ \sqrt{p\pi} }{\sqrt{2}\sigma^3} \right)^p \right),
    \]
    when the Halton sequence is used to perform QMC and $\gamma^{\prime}_n$ with $t(n)^2 = \log(n)$ was used to truncate the integral.
\end{proposition}
This is an improvement in the rate compared to Corollary \ref{cor:QMC-HD}. In practice, should the curvature i.e. $\sigma$ be very small of order $1/p^{1/6}$ for example, then this will make $\left(1 +  \sqrt{p\pi/ 2\sigma^6 } \right)^p  = O(C^p)$, where the increased variance disperses the mass of the normal distribution, making it very similar locally, and consequentially delaying the curse of dimensionality.

\begin{proof}
First we verify Assumptions \ref{assn:delta_decay}--\ref{assn:partial}. Assumption \ref{assn:delta_decay} holds by the strict log-concavity of the log-likelihood, and $\epsilon$ can be taken to be $1$. 
For Assumption \ref{assn:hess}, the Hessian of the log-likelihood is $n I_p/\sigma^2$, therefore we can take $\eta_1 = \eta_2 = 1/\sigma^2$.
Finally for Assumption \ref{assn:partial}, all derivatives of the log-likelihood of order 3 or greater are $0$ and all cross derivatives are also $0$ by independence, therefore it remains to bound the normalized first and second order derivatives for $\mu \in [n\bar{y}/(n+1) - \gamma_n, n\bar{y}/(n+1) + \gamma_n]$ where $\bar{y}$ is the average of the observations $y_1, \dots, y_n$. The first order derivatives are
\*[
   \frac{1}{n} \frac{\partial}{\partial \mu_j} \left( -\frac{1}{2\sigma}\sum_{i = 1}^n\lVert y_i -\mu\rVert^2_2 - \frac{\lVert \mu \rVert_2^2}{2\sigma_p} \right) = \frac{1}{\sigma n} \sum_{i = 1}^n (y_{ij} - \mu_j) - \frac{\mu_j}{\sigma_p n},
\]
where the first term 
\*[
    \frac{1}{\sigma n} \sum_{i = 1}^n (y_{ij} - \mu_j) \leq \frac{1}{\sigma} \gamma_n,
\]
for all indices $j \leq p$ and as $p \rightarrow \infty$, this bound holds with probability tending to $1$ by standard sub-Gaussian concentration arguments should $p/n \rightarrow 0$. The union bound can then be applied to make this hold uniformly over all indices $j$. The second term deterministically tends to $0$ at a rate of $\gamma_n/n$.

We now modify the proof of the HK bound in Lemma \ref{lem:hk-var}. The crucial difference is in case 2 in \eqref{eq:case2} which can be replaced by 
\*[
     \left| \sum_{A \in \Pi} \prod_{B \in A} \frac{\partial^{|B|}l_n(x_\alpha, 1_{-\alpha})}{\prod_{j \in B} \partial\theta_{j} } \right| \leq D^k n^k,
\]
where $\Pi$ is the collection of all possible partitions over the set $(i_1,\dots, i_k)$, $B$ is the collection of all blocks of a partition $A$ -- since if any of the blocks $B$ contains more than a single element, then this cross derivative will be 0 by the independence of the components of the vector; therefore there is a single non-zero term within this sum. This results in \eqref{eq:case2-final} being changed to
\*[
    \sum_{i = 2}^{p - 1}&\sum_{|\alpha| = k} \int_{[0,1]^k} \frac{\partial^{|\alpha|}}{\partial x_{\alpha}} f(x_\alpha, 1_{-\alpha}) dx_{\alpha} \\
    &\leq \sum_{i = 2}^{p - 1} {{p}\choose{k}} \frac{ D^k p^{k/2} (\eta_2\pi)^{k/2} }{2^{k/2}n^{p(p - k) - k}} \leq \frac{1}{n} \sum_{i = 2}^{p - 1}{{p}\choose{k}} \left( \frac{ D \sqrt{p\eta_2\pi} }{\sqrt{2}} \right)^k \leq \frac{1}{n} \left(1 + \frac{ D \sqrt{p\eta_2\pi} }{\sqrt{2}} \right)^p,
\]
by the binomial theorem.
Similarly, for case 3 in \eqref{eq:case3} there is only a single partition where every element is partitioned such that each term is alone, therefore we have that a single term of size
\*[
(D\sqrt{\pi})^p \prod_{i = \lfloor p/2 \rfloor}^{p} i, 
\]
which shows the desired result after using the Koksma-Hlawka inequality.
\end{proof}

\begin{remark}
    The same calculations can be performed for an arbitrary multivariate normal distribution, in which case all derivatives of order 2 or greater will then be $0$, and a similar combinatorial argument can then be used to upper bound the HK variation.
\end{remark}

\section{Mixed Models}\label{sec:mixed_models}
Monte Carlo integration methods can be applied to mixed models in which the integrated likelihood is of interest
\*[
L_n(\theta; Y_n) = \int_{U} \pi(u) \exp( l_n(\theta, u; Y_n) ) du.
\]
The results of the previous sections are readily applicable for MC or QMC approximations to this function, for any fixed value of $\theta$, Assumptions \ref{assn:delta_decay} to \ref{assn:partial} can be checked and the rates from the relevant theorem would apply. 

For practical applications the numerical approximation needs to be accurate over a set of values, but supposing that Assumptions \ref{assn:delta_decay}--\ref{assn:prior} hold uniformly over a set around the true parameter, then we can obtain a guarantee on the quality of the MMLE obtained from this approximated likelihood.

Consider a generalized linear mixed model
\begin{align*}
    y_{ij} \mid x_{ij}, v_{ij}, u_i &\overset{\text{ind}}{\sim} F(\eta_{ij}), \notag \\
    \eta_{ij} &= h(\mu_{ij}) = \beta_0 + x_{ij}^{\top} \beta + v_{ij}^{\top} u_i, \notag \\
    u_i &\overset{\text{ind}}{\sim} N(0, \Sigma),
\end{align*}
where $i \in \{1, \dots, k\}$ indicates group partnership with $k$ being the total number of `groups', $j \in \{1, \dots, n_j\}$ indicates the individual within each group, $n_j$ is the number of observations per group $i$, $\beta$ are the fixed effects, $u$ are the random effects, and $x_{ij}$ and $v_{ij}$ are covariates.
As the observations are independent across groups, the marginal likelihood can be written as
\[\label{eq:MMLE}
L_n(\theta; Y_n) = \prod_{i = 1}^k \int_{\Theta} l_{n_i}(y_i, \mu_i; \theta) d\mu_i,
\]
and each integral can be approximated individually.

In the theorem to follow, we let $\hat\theta_n$ denote the true MMLE, $\tilde{\theta}_n$ the MMLE obtained from the approximate likelihood, $m$ the number of grid points used for each individual integral, $p$ the maximum dimension of any single integral, and $n$ the largest number of observations in any single group; we also assume that the truncation for MC is performed with the same $\gamma_n^2 = \eta_1\log(n)/\eta_2$ and the truncation for QMC is performed with $(\gamma_n^\prime)^2 = \eta_1p\log(n)/\eta_2$, where $\eta_1$ and $\eta_2$ are taken from Assumption 2 in \cite{stringer2022fitting}; these quantities will need to be estimated in practice.

\begin{theorem}
    Under Assumptions 1--7 of \cite{stringer2022fitting}, and assuming that the true MMLE satisfies
    \*[
    r_n (\hat\theta_n - \theta_0) \xrightarrow{D} Z,
    \]
    for some sequence $r_n \rightarrow \infty$ as $n \rightarrow \infty$ and random variable $Z$, then the following holds for $\tilde\theta_n$ when using an uniformly distributed random grid with a truncation level determined by $\gamma_n$ with $t(n)^2 = \eta_2\log(n)/\eta_1$ 
    \*[
    r_n(\tilde\theta_n - \theta_0) =Z +o_p(1) + O\left(r_n \left\{\frac{\log(n)^{p +1} n^{ 3/2 +\epsilon} 
\log(m)}{m} \right\}^{1/4} \right),
    \]
    for every $\epsilon > 0 $. While for QMC with a low discrepancy grid,
    \*[
     r_n(\tilde\theta_n - \theta_0) =Z +o_p(1) + O\left(r_n \left\{ \frac{ n^{1 + \epsilon} \log(n)^{p/2} \log(m)^{p }}{m}  \right\}^{1/2} \right),
    \]
    for every $\epsilon >0$ if $\gamma^{\prime}_n$ with $t(n)^2 = \eta_2\log(n)/\eta_1$ is used to truncate the integral.
\end{theorem}
Assumptions 1--7 required above impose uniformity on the derivates of the log-likelihood function and imply our Assumptions \ref{assn:delta_decay}--\ref{assn:partial}, hence why we can use the same truncation for all integrals in \eqref{eq:MMLE}.
The proof of this Theorem essentially follows from substituting the error rates of MC and QMC into Equation (3.2) of \cite{stringer2022fitting} and following the same arguments. Typically $Z$ follows a multivariate normal distribution \citep{jiang2022usable,jiang2024preciseasymptoticslinearmixed} under certain scaling limit requirements on the number of observations per group and the number of groups tends to infinity.


\section{Simulations}\label{sec:sims}
To test the performance of QMC and MC, we consider a simple example involving the normalizing constant of a multivariate normal model as the dimension increases given in Section \ref{sec:normal}, where $X \sim N(\mu, \Sigma)$ for which the mean is a vector of $0$'s and the covariance structure is $\Sigma = I_p$. We assume the prior to have the same mean and covariance as the observations.

We consider dimensions ranging from $p = 1, 2, 4, 8$, with the number of samples from $n = 8, 16, 32, 64$ and the number of integration points used ranging from $m = 400,800,1600,3200$. For each configuration, $1000$ replicates were simulated to obtain some empirical confidence intervals for MC.

Tables \ref{tab:1}--\ref{tab:4} provide the relative and the absolute integration errors, each figure holding the dimension constant while varying $m$ and $n$.
In this example the normalizing constant does not depend on data, except the posterior mode where the truncation is centered around, therefore QMC does not have much variability, as the grid was taken to be deterministic.
As expected, performance degrades as dimension increases, a downward bias in the mean relative accuracy which arises from the truncation argument since the integrand is positive, so truncation will remove some mass.

We note that MC tends to do better on average than QMC, however when examining the 95\% confidence interval, QMC tends to be well within the middle of the interval rather than the extremes in Table \ref{tab:3} for example, but this is most apparent in \ref{tab:4} when $p=8$, where MC can provide performance as bad a $\approx 99.95\%$ underestimation in row 13 while QMC gives a $\approx 91.56\%$ underestimation of the normalizing constant.

\begin{table}[ht]
\centering
\begin{tabular}{rrrrrrrr}
  \hline
 & p & n & m & Mean MC & 2.5\% MC & 97.5\% MC & Mean QMC \\ 
  \hline
1 & 1 & 8 & 400 & -0.125982 & -0.153374 & -0.099533 & -0.125993 \\ 
  2 & 1 & 8 & 800 & -0.126218 & -0.145385 & -0.107493 & -0.126104 \\ 
  3 & 1 & 8 & 1600 & -0.126335 & -0.139115 & -0.112837 & -0.126132 \\ 
  4 & 1 & 8 & 3200 & -0.126235 & -0.136057 & -0.117153 & -0.126138 \\ 
  5 & 1 & 16 & 400 & -0.086792 & -0.121971 & -0.053978 & -0.085939 \\ 
  6 & 1 & 16 & 800 & -0.085430 & -0.111019 & -0.061461 & -0.086059 \\ 
  7 & 1 & 16 & 1600 & -0.086590 & -0.103454 & -0.068063 & -0.086088 \\ 
  8 & 1 & 16 & 3200 & -0.086273 & -0.098045 & -0.074845 & -0.086095 \\ 
  9 & 1 & 32 & 400 & -0.058232 & -0.097963 & -0.018753 & -0.058531 \\ 
  10 & 1 & 32 & 800 & -0.058703 & -0.084621 & -0.029519 & -0.058652 \\ 
  11 & 1 & 32 & 1600 & -0.058622 & -0.077740 & -0.037699 & -0.058680 \\ 
  12 & 1 & 32 & 3200 & -0.058873 & -0.073355 & -0.044722 & -0.058687 \\ 
  13 & 1 & 64 & 400 & -0.040983 & -0.086923 & 0.006367 & -0.039708 \\ 
  14 & 1 & 64 & 800 & -0.039621 & -0.072574 & -0.005467 & -0.039824 \\ 
  15 & 1 & 64 & 1600 & -0.039285 & -0.064280 & -0.015697 & -0.039851 \\ 
  16 & 1 & 64 & 3200 & -0.039527 & -0.056466 & -0.022456 & -0.039857 \\ 
   \hline
\end{tabular}
\caption{Relative approximation errors for MC and QMC for $p = 1$.}\label{tab:1}
\end{table}
\begin{table}[ht]
\centering
\begin{tabular}{rrrrrrrr}
  \hline
 & p & n & m & Mean MC & 2.5\% MC & 97.5\% MC & Mean QMC \\ 
  \hline
1 & 2 & 8 & 400 & -0.059646 & -0.129644 & 0.017436 & -0.056819 \\ 
  2 & 2 & 8 & 800 & -0.059564 & -0.113554 & -0.009089 & -0.057973 \\ 
  3 & 2 & 8 & 1600 & -0.059759 & -0.097967 & -0.020003 & -0.059975 \\ 
  4 & 2 & 8 & 3200 & -0.060711 & -0.086297 & -0.033685 & -0.060136 \\ 
  5 & 2 & 16 & 400 & -0.031025 & -0.126524 & 0.058450 & -0.025944 \\ 
  6 & 2 & 16 & 800 & -0.030583 & -0.095416 & 0.036938 & -0.027276 \\ 
  7 & 2 & 16 & 1600 & -0.030507 & -0.077855 & 0.017280 & -0.030106 \\ 
  8 & 2 & 16 & 3200 & -0.030091 & -0.063221 & 0.003664 & -0.030202 \\ 
  9 & 2 & 32 & 400 & -0.016414 & -0.130159 & 0.101234 & -0.009653 \\ 
  10 & 2 & 32 & 800 & -0.014486 & -0.091230 & 0.064283 & -0.011323 \\ 
  11 & 2 & 32 & 1600 & -0.014516 & -0.067956 & 0.043124 & -0.014956 \\ 
  12 & 2 & 32 & 3200 & -0.014922 & -0.053751 & 0.024186 & -0.014958 \\ 
  13 & 2 & 64 & 400 & -0.008859 & -0.129531 & 0.118252 & -0.000824 \\ 
  14 & 2 & 64 & 800 & -0.008666 & -0.095414 & 0.085371 & -0.003026 \\ 
  15 & 2 & 64 & 1600 & -0.007485 & -0.071800 & 0.055850 & -0.007394 \\ 
  16 & 2 & 64 & 3200 & -0.007334 & -0.053565 & 0.039549 & -0.007274 \\ 
   \hline
\end{tabular}
\caption{Relative approximation errors for MC and QMC for $p = 2$.}
\end{table}

\begin{table}[ht]
\centering
\begin{tabular}{rrrrrrrr}
  \hline
 & p & n & m & Mean MC & 2.5\% MC & 97.5\% MC & Mean QMC \\ 
  \hline
1 & 4 & 8 & 400 & -0.008929 & -0.266980 & 0.278155 & 0.045932 \\ 
  2 & 4 & 8 & 800 & -0.006475 & -0.189544 & 0.180882 & 0.008732 \\ 
  3 & 4 & 8 & 1600 & -0.010506 & -0.142256 & 0.134897 & -0.012457 \\ 
  4 & 4 & 8 & 3200 & -0.006877 & -0.099015 & 0.093834 & -0.015314 \\ 
  5 & 4 & 16 & 400 & -0.002282 & -0.325366 & 0.396144 & 0.086599 \\ 
  6 & 4 & 16 & 800 & 0.000427 & -0.229492 & 0.266101 & 0.027705 \\ 
  7 & 4 & 16 & 1600 & -0.001236 & -0.178641 & 0.180994 & -0.006613 \\ 
  8 & 4 & 16 & 3200 & -0.002734 & -0.124856 & 0.117629 & -0.012749 \\ 
  9 & 4 & 32 & 400 & 0.000894 & -0.393482 & 0.511684 & 0.118669 \\ 
  10 & 4 & 32 & 800 & -0.002124 & -0.273851 & 0.325494 & 0.040281 \\ 
  11 & 4 & 32 & 1600 & 0.001899 & -0.210273 & 0.223822 & -0.006128 \\ 
  12 & 4 & 32 & 3200 & 0.000415 & -0.138241 & 0.167645 & -0.015105 \\ 
  13 & 4 & 64 & 400 & -0.005731 & -0.468204 & 0.596599 & 0.140998 \\ 
  14 & 4 & 64 & 800 & -0.006463 & -0.345192 & 0.402549 & 0.047082 \\ 
  15 & 4 & 64 & 1600 & -0.000577 & -0.227556 & 0.275430 & -0.008002 \\ 
  16 & 4 & 64 & 3200 & -0.000274 & -0.185032 & 0.200968 & -0.019148 \\ 
   \hline
\end{tabular}
\caption{Relative approximation error for MC and QMC for $p = 4$.}
\label{tab:3}
\end{table}

\begin{table}[ht]
\centering
\begin{tabular}{rrrrrrrr}
  \hline
 & p & n & m & Mean MC & 2.5\% MC & 97.5\% MC & Mean QMC \\ 
  \hline
1 & 8 & 8 & 400 & 0.084861 & -0.953905 & 4.936026 & -0.465240 \\ 
  2 & 8 & 8 & 800 & 0.025032 & -0.891627 & 3.540811 & -0.414345 \\ 
  3 & 8 & 8 & 1600 & 0.046622 & -0.790689 & 2.524881 & -0.442301 \\ 
  4 & 8 & 8 & 3200 & -0.002301 & -0.716486 & 1.639213 & -0.137481 \\ 
  5 & 8 & 16 & 400 & -0.170066 & -0.990298 & 4.704067 & -0.673619 \\ 
  6 & 8 & 16 & 800 & 0.017445 & -0.957049 & 4.610040 & -0.643797 \\ 
  7 & 8 & 16 & 1600 & 0.011627 & -0.922310 & 4.030156 & -0.616550 \\ 
  8 & 8 & 16 & 3200 & 0.046587 & -0.861213 & 3.451607 & -0.291093 \\ 
  9 & 8 & 32 & 400 & 0.088609 & -0.997090 & 7.635046 & -0.824456 \\ 
  10 & 8 & 32 & 800 & -0.014428 & -0.989544 & 4.928271 & -0.811912 \\ 
  11 & 8 & 32 & 1600 & -0.115838 & -0.970739 & 3.956533 & -0.763803 \\ 
  12 & 8 & 32 & 3200 & 0.106969 & -0.929065 & 4.795452 & -0.467314 \\ 
  13 & 8 & 64 & 400 & 0.197498 & -0.999515 & 13.943347 & -0.915594 \\ 
  14 & 8 & 64 & 800 & 0.059107 & -0.997255 & 7.343841 & -0.912046 \\ 
  15 & 8 & 64 & 1600 & 0.020557 & -0.990012 & 7.198481 & -0.868250 \\ 
  16 & 8 & 64 & 3200 & 0.007722 & -0.966710 & 4.848322 & -0.629245 \\ 
   \hline
\end{tabular}
\caption{Relative approximation error for MC and QMC for $p = 8$.}
\label{tab:4}
\end{table}

\section{Discussion}\label{sec:discussion}
Bounds for Monte Carlo and quasi-Monte Carlo integration of the normalizing constant and the marginal likelihood are derived accounting for $n, p$ and $m$. The bounds for QMC display much worse dependency in dimensionality than for MC. Some considerations for potential future work: 

\begin{itemize}
    \item Sequential QMC. It is a commonly known fact that QMC and MC suffer from the curse of dimensionality, and this work provides a more exact statement on how badly dimension affects the estimation error. A logical next step would be to examine variants of sequential Monte Carlo \citep{del2006sequential} which were designed to alleviate this problem.
    \item Discrepancy bounds for other low discrepancy sequences as $p \rightarrow \infty$. The star discrepancy for the Halton sequence was derived in high dimensions, but it is possible that other sequences have better scaling in terms of rate or constants.
    Comparing the discrepancy of these sequences in this new regime can shed more light on the optimality of a given sequence. 
    \item Randomized quasi-Monte Carlo. Our work is limited to QMC, but randomized QMC is a popular alternative to using strictly deterministic sequences. In particular, CLT results are available for RQMC due to the randomization, for example it is known that under additional continuity assumption on the $p$-th order partial derivatives, using a particular scrambled net gives \citep{owen2019monte}
    \*[
    P\left( \frac{\hat\mu - \mu}{\sqrt{\text{Var}(\hat\mu) }} \leq z \right) \rightarrow \Phi(z),
    \]
    however it is no longer clear if a result of this kind holds as $p \rightarrow \infty$.
    \item The bounds on the truncation error of the integral show that we only need to truncate the integral in a $L^2$ ball instead of a cube, however the bounds we used to bound the integration errors for MC and QMC require that grid to be on a square region. This was needed for MC as the uniform distribution on a cube has independent components, while the standard formulation of QMC is for a function defined on a hypercube. At the very least, this would lead to an improvement of a $\sqrt{p}$-factor in the MC bounds in high dimensions.
    \item We assume that the region of interest where the mass of the function lies is a hypercube, however it may be possible that the mass of the functions concentrates unevenly in hyper-rectangles whose sides may decrease at different rates. This case can be handled by considering Hessians where some derivatives along key directions grow slower or quicker than others, in doing so it maybe be possible to quantify the directions along which the integration error grows quicker and to implement more adaptive type methods.
\end{itemize}

\begin{appendix}
\section{Technical Lemmas for MC}\label{app:MC}
In order to control the truncation error, we modify the approach taken in Lemma C.1 from the supplement of \cite{tang2021laplace}. 
\begin{lemma} 
    \label{lemma:A-trunc}
    Under Assumption \ref{assn:hess}, for $\gamma^2_n = t(n)/(\eta_1n)$, for any sequence $t(n)\rightarrow \infty$ and any fixed $p$, we have
    \*[
        \int_{ B^C_{\maximizer}(\gamma_n) \cap B_{\maximizer}(\delta_{n,p}) }  \exp\{ l_n(\theta; X_n)  - l_n(\maximizer ; X_n)  \} d\theta &= O(n^{-p/2}\exp\{pt(n)/(2 + \epsilon)\})\\
    \frac{\int_{ B^C_{\maximizer}(\gamma_n) \cap B_{\maximizer}(\delta_{n,p}) }  \exp\{ l_n(\theta; X_n)  - l_n(\maximizer ; X_n)  \} d\theta}{\int_{ \mathbb{R}^d }  \exp\{ l_n(\theta; X_n)  - l_n(\maximizer ; X_n)  \} d\theta}  &=
    O(\exp\{pt(n)/(2 + \epsilon)\}),
    \]
    for every $\epsilon > 0$.
    Under the same assumptions with $p \rightarrow \infty$, and $(\gamma^\prime_n)^2 = t(n)p/(\eta_1n)$
    \*[
        \int_{ B^C_{\maximizer}(\gamma_n) \cap B_{\maximizer}(\delta_{n,p}) }  \exp\{ l_n(\theta; X_n)  - l_n(\maximizer ; X_n)  \} d\theta &= O(n^{-p/2}\exp\{pt(n)/(2 + \epsilon)\})\\
    \frac{\int_{ B^C_{\maximizer}(\gamma_n) \cap B_{\maximizer}(\delta_{n,p}) }  \exp\{ l_n(\theta; X_n)  - l_n(\maximizer ; X_n)  \} d\theta}{\int_{ \mathbb{R}^d }  \exp\{  l_n(\theta; X_n)  - l_n(\maximizer ; X_n)  \} d\theta}  &=
    O\left\{ \left(\frac{\eta_1}{\eta_2}\right)^p \exp\{pt(n)/(2 + \epsilon)\right\}.
    \]
    \end{lemma}
\begin{proof}
    We first upper bound the numerator. Let $A = B_{\zerop}(\delta_{n,p})$, and $D = B_{\zerop}(\gamma_n)$
    \[
    & \int_{ B^C_{\maximizer}(\gamma_n) \cap B_{\maximizer}(\delta_{n,p}) }  \exp\{ l_n(\theta' ; X_n)  - l_n(\maximizer ; X_n)  \} d\theta' \nonumber \\
    &\leq \int_{ D^C \cap A }  \exp\left\{-\frac{1}{2}\theta^\top l_n^{(2)}(\tilde{\theta}) \theta   \right\} d\theta \label{eq:lemma1_lp},
    \]
    by a change of variable $\theta = \theta' - \maximizer$ and where $\tilde{\theta} = \tau(\theta)\theta + \{1 - \tau(\theta) \} \maximizer$, for $0 \leq \tau(\theta) \leq 1$. By Assumption \ref{assn:hess}, 
    \*[
    (\ref{eq:lemma1_lp})  &\leq    \int_{A \cap D^C}  \exp\left(-\frac{\eta_1 n}{2} \theta^\top \theta   \right)  d\theta \\
    &\leq\int_{D^C} (2\pi)^{p/2}  \det\{ \eta_1 I_p/n \}^{1/2}  \phi(\theta ; 0, \eta_1 I_p/n)  d\theta \\
    &= \left(\frac{2\pi\eta_1}{n} \right)^{p/2} \mathbb{P} \left[\chi^2_p \geq n \eta_1 \gamma_n^2  \right] =  \left(\frac{2\pi\eta_1}{n} \right)^{p/2} P\left[\chi^2_p/p \geq 1+ \zeta_n \right],
    \]
    where $\zeta_n = n\gamma^2_n \eta_1/p - 1$, and the region of integration was changed to a larger one by using $D^C$ instead of $A \cap D^C$. By Lemma 3 in \cite{Fan},
    \begin{align*}
        P\left[\chi^2_p/p \geq 1+ \zeta_n \right] \leq \exp\left[ \frac{p}{2} \{ \log(1 +\zeta_n) - \zeta_n \}  \right].
    \end{align*}
    Now, $n\gamma_n^2\eta_1/p = \log(n) \rightarrow \infty$, 
    so there exists $\epsilon > 0$ and $N_0$ such that $ \log(1 +\zeta_n) - \zeta_n  \leq - t(n)/(1+\epsilon) $ for all $n > N_0$ which implies
    \begin{align*}
         P\left[\chi^2_p/p \geq 1+ \zeta_n \right] \leq  \exp\{ -pt(n)/2(1 + \epsilon) \},
    \end{align*}
    for an $\epsilon > 0$ which can be made arbitrarily small.
    The same argument holds for $\gamma^\prime_n$.
    
    We now lower bound the denominator:
    \begin{align}
        &\int_{ \mathbb{R}^d }  \exp\{ l_n(\theta; X_n)  - l_n(\maximizer ; X_n)  \} d\theta \nonumber\\
        &\geq \int_{ B_{\hat\theta_n}(\delta_{n,p}) }  \exp\{ l_n(\theta; X_n)  - l_n(\maximizer ; X_n)  \} d\theta \nonumber\\
        &\geq \int_{B_{\hat\theta_n}(\delta_{n,p})}  \exp\left(-\frac{\eta_2 n}{2} \theta^\top \theta   \right)  d\theta\nonumber\\
        &\geq \int_{B_{\hat\theta_n}(\delta_{n,p})} (2\pi)^{p/2}  \det\{ \eta_1 I_p/n \}^{1/2}  \phi(\theta ; 0, \eta_2 I_p/n) d\theta \geq \frac{1}{2} \left(\frac{2\pi\eta_2}{n} \right)^{p/2}. \label{eq:lower_bound}
    \end{align}
    Combining this lower bound with the upper bound gives the desired result.
\end{proof}    

We bound the Lipschitz constant of the likelihood function within the truncated region:

\begin{lemma}\label{lem:lip}
    For $\gamma^2_n = t(n)/(\eta_1n)$ and $(\gamma^\prime_n)^2 = t(n)p/(\eta_1n)$ the Lipschitz constant of the function $\exp\{ l_n(\theta; Y_n) - l_n(\hat\theta_n; Y_n)\}$ is upper bounded by $\eta_2 (t(n) p /\eta_1)^{1/2}$ and $ p \eta_2 (t(n)/\eta_1)^{1/2}$ respectively on the grid $[\hat\theta_n -\gamma_n, \hat\theta_n +\gamma_n]^p$.
\end{lemma}

\begin{proof}
    For all values of $\theta, \theta^\prime \in B^\infty_{\hat\theta_n}(\gamma_n) $ by a first-order Taylor expansion and noting that the operator norm for a scalar and vector is simply its $L^2$ norm,
    \[
        &\norm{\exp\{ l_n(\theta; Y_n) - l_n(\hat\theta_n; Y_n)\} - \exp\{ l_n(\theta^\prime; Y_n) - l_n(\hat\theta_n; Y_n)\}}_{op} \nonumber\\
        &= \norm{ (\theta - \theta^\prime)^\top \frac{\partial}{\partial\theta} l_n(\tilde{\theta}; Y_n) \exp\{ l_n(\tilde{\theta}; Y_n) - l_n(\hat\theta_n; Y_n)\}}_{op}\nonumber\\
        &=\norm{(\theta - \theta^\prime)^\top \frac{\partial^2}{\partial\theta\partial\theta^T} l_n(\tilde{\theta}^\prime; Y_n) (\tilde\theta - \hat\theta_n) \exp\{ l_n(\tilde{\theta}; Y_n) - l_n(\hat\theta_n; Y_n)\}}_{op}\nonumber\\
        &\leq \norm{\theta - \theta^\prime}_2 \norm{\tilde\theta - \hat{\theta}_n}_2 \norm{\frac{\partial^2}{\partial\theta\partial\theta^T} l_n(\tilde{\theta}^\prime; Y_n)}_{op} \exp\{ l_n(\tilde{\theta}; Y_n) - l_n(\hat\theta_n; Y_n)\} \label{eq:lemma_lip1}.
    \]
    We now bound each term in \eqref{eq:lemma_lip1} separately. Note that $\theta, \theta^\prime, \tilde\theta, \tilde\theta^\prime$ all lie within $B^\infty_{\hat\theta_n}(\gamma_n)$ hence
    \[
         \norm{\tilde\theta - \hat{\theta}_n}_2 \leq p^{1/2}\gamma_n \label{eq:lemma_lip2}\\
        \norm{\frac{\partial^2}{\partial\theta\partial\theta^T} l_n(\tilde{\theta}^\prime; Y_n)}_{op} \leq \eta_2 n,\label{eq:lemma_lip3}\\
        \exp\{ l_n(\tilde{\theta}; Y_n) - l_n(\hat\theta_n; Y_n)\} \leq 1 \label{eq:lemma_lip4}
    \]
    combining \cref{eq:lemma_lip2,eq:lemma_lip3,eq:lemma_lip4} gives
    \*[
    \cref{eq:lemma_lip1} \leq \norm{\theta - \theta^\prime}_2 \frac{ \eta_2\{npt(n)\}^{1/2} }{\sqrt{\eta_1}},
    \]
    for $\gamma_n$ while for $\gamma_n^\prime$ we can repeat the argument to obtain
    \*[
    \cref{eq:lemma_lip1} \leq \norm{\theta - \theta^\prime}_2 \frac{ \eta_2\{nt(n)\}^{1/2} p }{\sqrt{\eta_1}}.
    \]
\end{proof}

\section{Technical lemmas for QMC}\label{app:QMC}

\begin{lemma}\label{lem:hk-var}
    Under Assumptions \ref{assn:hess} and \ref{assn:partial}, using $\gamma_n^\prime$ with $t(n) = \eta_1 \log(n)/\eta_2$ and $p\geq 3$, then there exists an $N_0$ such that for every $n \geq N_0$
    \*[
        V_{HK}(L_n(\theta)) &\leq (2\gamma_n^\prime)^p \left\{ \frac{D \sqrt{\eta_2p^3\log(n)} }{n^{p(p-1)-1/2} } + B_{p} \frac{D^{(p-1)} (\eta_2\pi p)^{(p+2)/2} }{2^{(p-1)/2}n}+(D\sqrt{\pi})^p \prod_{i = \lfloor p/2 \rfloor}^{p} i \right\},
    \] 
    when the likelihood function is restricted on the grid $[\hat\theta_n -\gamma_n, \hat\theta_n + \gamma_n]^p$. For $p$-fixed, this implies that:
    \*[
        V_{HK}(L_n(\theta)) = O(\log(n)^{p/2} n^{-p/2}),
    \]
    whereas if we allow $p \rightarrow \infty$
    \*[
        V_{HK}(L_n(\theta)) = O\left( \frac{(C')^{(p -1)/2} p^{ (3p + 2)/2 }}{n\log(p+1)^p} \right),
    \]
    for some constant $C > 0$.
\end{lemma}

\begin{proof}
    We bound the HK variation using 
    \*[
    V_{HK}(f) \leq (2\gamma_n^\prime)^p \sum_{\alpha \subset \mathcal{P}\{1,2,\dots, p\}-  \emptyset} \int_{[0,1]^p} \frac{\partial^{|\alpha|}}{\partial x_{\alpha}} f(x_\alpha, 1_{-\alpha}) dx_{\alpha}.
    \]
    We bound the terms in the sum on case-by-case basis based on the cardinality of $\alpha$. 
    \\
    \textbf{Case 1:} First consider if $|\alpha| = 1$, by the chain rule
    \*[
        &\sum_{i = 1}^p \int_{[0,1]} \frac{\partial}{\partial x_i} f(x_i, 1_{-i}) dx_{i} \\
        &= \sum_{i = 1}^p \gamma_n \int_{[0,1]} \frac{\partial l_n\{2\gamma_n(x_i -1/2), \gamma_n\dot1_{-i}\}}{\partial x_i} \\
        \quad &\times \exp\{ l_n\{2\gamma_n(x_i -1/2) + \hat\theta_{n,i}, \gamma_n\dot1_{-i}+ \hat\theta_n\} - l_n\{\hat\theta_{n,i}, \hat\theta_n\}\} dx_{i}\\
        &\leq \sum_{i = 1}^p \gamma_n \exp\left\{ -p(p - 1)\log(n)  \right\} \int_{[0,1]} \frac{\partial l_n\{2\gamma_n(x_i -1/2), \gamma_n\dot1_{-i}\}}{\partial x_i} dx_{i}\\
        &\leq D p^{3/2}n^{1/2}\eta_2^{1/2}\log(n)^{1/2} \exp\left\{ -p(p - 1)\log(n)  \right\}\\
        &= \frac{D \sqrt{\eta_2p^3\log(n)} }{n^{p(p-1)-1/2} },
    \]
    by Lemma \ref{lemma:exp_diff} and Assumption \ref{assn:partial}.

    \textbf{Case 2:} Should $p \geq 3$, consider any $\alpha$ such that $|\alpha| = k$ and $1 < |\alpha| \leq p - 1$ ,
    \*[
        \sum_{|\alpha| = k} \int_{[0,1]^k} \frac{\partial^{|\alpha|}}{\partial x_{\alpha}} f(x_\alpha, 1_{-\alpha}) dx_{\alpha} \leq {{p}\choose{|\alpha|}} \max_{|\alpha| = k} \int_{[0,1]^k} \frac{\partial^{|\alpha|}}{\partial x_{\alpha}} f(x_\alpha, 1_{-\alpha}) dx_{\alpha}.
    \]
    From the multivariate Faa di Bruno formula for every $k < p$
    \begin{align*}
        \frac{\partial^k }{\partial\theta_{i_1}\cdots \partial\theta_{i_k} } f(x_\alpha, 1_{-\alpha}) &= (\gamma_n^\prime)^{k} \sum_{A \in \Pi} \exp\{l_n(x_\alpha, 1_{-\alpha}) - l_n(x_\alpha, 1_{-\alpha})\} \prod_{B \in A} \frac{\partial^{|B|}l_n(x_\alpha, 1_{-\alpha})}{\prod_{j \in B} \partial\theta_{j} }\\
        &\leq (\gamma_n^\prime)^{k} \exp\left\{ -\frac{\eta_2t(n)(p - |\alpha|)}{\eta_1}  \right\} \sum_{A \in \Pi} \prod_{B \in A} \frac{\partial^{|B|}l_n(x_\alpha, 1_{-\alpha})}{\prod_{j \in B} \partial\theta_{j} } ,
    \end{align*}
    where $\Pi$ is the collection of all possible partitions over the set $(i_1,\dots, i_k)$, $B$ is the collection of all blocks of a partition $A$, and the inequality follows from Lemma \ref{lemma:exp_diff}. Noting the order of the product term increases as the number of blocks increases, we have
    \[\label{eq:case2}
        \left| \sum_{A \in \Pi} \prod_{B \in A} \frac{\partial^{|b|}l_n(x_\alpha, 1_{-\alpha})}{\prod_{j \in B} \partial\theta_{j} } \right| \leq  B_k D^k n^{k},
    \]
    where $B_k$ is the $k$-th Bell number, which counts the number of total of possible partitions of a set of size $k$. Meaning that
    \[\label{eq:case2-final}
    &\sum_{|\alpha| = k} \int_{[0,1]^k} \frac{\partial^{|\alpha|}}{\partial x_{\alpha}} f(x_\alpha, 1_{-\alpha}) dx_{\alpha} \\
    &\leq {{p}\choose{k}} (\gamma_n^\prime)^{k} \exp\left\{ -p(p - k)\log(n)  \right\} 
     B_k (Dn)^{k} \int_{[0,1]^k} \exp\left( - 2\log(n) \left\lVert x_\alpha - \frac{1}{2} \right\rVert_2^2 \right) d X_\alpha\nonumber\\
    &\leq {{p}\choose{k}} \frac{D^k n^k (\eta_2\log(n) p)^{k/2} B_k}{n^{p(p - k)}} \left(\frac{\pi}{2\log(n)} \right)^{k/2} = {{p}\choose{k}} \frac{ D^k p^{k/2} B_k (\eta_2\pi)^{k/2} }{2^{k/2}n^{p(p - k) - k}},\nonumber
    \]
    therefore
    \*[
       \sum_{k = 2}^{p-1} \sum_{|\alpha| = k} \int_{[0,1]^k} \frac{\partial^{|\alpha|}}{\partial x_{\alpha}} f(x_\alpha, 1_{-\alpha}) dx_{\alpha} &\leq \sum_{k = 2}^{p-1} {{p}\choose{k}} \frac{D^k B_k (\eta_2 p\pi)^{k/2} }{2^{k/2}n^{p(p - k) - k}}\\
       &\leq  B_{p} \frac{D^k (\eta_2\pi p)^{(p+2)/2} }{2^{(p-1)/2}n},
    \]
    where the last inequality follows from the identity:
    \*[
       \sum_{k = 2}^{p-1} {{p}\choose{k}} B_k \leq p \sum_{k = 1}^{p - 1} {{p-1}\choose{k}} B_k = p B_{p}.
    \]
    \textbf{Case 3:} Finally for $\alpha = \{1,2,\dots, p\}$
    \*[
        &\int_{[0,1]^p}  \frac{\partial^{p}}{\partial x_1 \partial x_2 \cdots x_p} f(\bold{x}) d\bold{x} \\
        &= \int_{[0,1]^p} \gamma_n^{p}  \exp\{l_n(2\gamma_n( x -1/2) + \hat\theta_n) - l_n(\hat\theta_n)\} \sum_{A \in \Pi} \prod_{B \in A} \frac{\partial^{|b|}l_n(x_\alpha, 1_{-\alpha})}{\prod_{j \in B} \partial\theta_{j}}     d\bold{x}
        \]
    where the term with the largest order is the partition which separates every variable into its own block:
    \[
         &\left|\int_{[0,1]^p} \prod_{i = 1}^p \frac{\partial l_n(2\gamma_n(x - 1/2 ) + \hat\theta_n)}{\partial x_i} \exp\left\{ l_n\{2\gamma_n(x -1/2) + \hat\theta_{n}\} - l_n(\hat\theta_{n}) \right\} d\bold{x} \right|\nonumber \\
         &\leq \int_{[0,1]^p} \left(2D \eta_2 p\log(n)  \right)^{p}  \left\lVert(x - 1/2) \right\rVert_2^p \exp\left\{ -2p\log(n)\lVert x - 1/2 \rVert_2^2  \right\} d\bold{x}\nonumber \]
         \newpage
         \[
         &\leq  \frac{\{\sqrt{\pi} 2D p\log(n)\eta_2\}^{p} }{\{4p\log(n) \}^{p/2}} \int_{[0,1]^p} \left\lVert(x - 1/2) \right\rVert_2^p \frac{\{4p\log(n) \}^{p/2}}{\pi^{p/2} } \exp\left\{ -2p\log(n)\lVert x - 1/2 \rVert_2^2  \right\} d\bold{x}\nonumber \\ \label{eq:case3}
         &\leq  \left( D\sqrt{\pi \eta_2 p\log(n)} \right)^p E_{X \sim N(1/2, 4p\log(n)I_p)}\left[\lVert X - 1/2 \rVert_2^p \cdot I\{X \in ([0,1]^p)^C\} \right]\nonumber \\
         &\leq (D\sqrt{\pi})^p \prod_{i = \lfloor p/2 \rfloor}^{p} i ,
    \]
    where we used Lemma \ref{lemma:chisqured} and the fact that
    \*[
        \left\lVert \frac{\partial l_n(2\gamma_n(x - 1/2 ) + \hat\theta_n)}{\partial x_i} \right\rVert_2 &= \left\lVert\gamma_n^\prime \frac{\partial l_n( \hat\theta_n)}{\partial \theta_i} + 2(\gamma_n^\prime)^{2} \frac{\partial^2 l_n( \tilde{x}+\hat\theta_n)}{\partial\bold{x}\partial x_i} (x - 1/2)\right\rVert_2 \\
        &\leq \left\lVert 2(\gamma_n^\prime)^{2}\right\rVert_2 \left\lVert \frac{\partial^2 l_n( \tilde{x}+\hat\theta_n)}{\partial\bold{x}\partial x_i}\right\rVert_2 \left\lVert(x - 1/2) \right\rVert_2 \leq 2D\eta_2 p\log(n)\left\lVert(x - 1/2) \right\rVert_2,
    \]
    since the derivative when evaluated at the posterior mode is $0$, meaning that
    \[
        \int_{[0,1]^p}  \frac{\partial^{p}}{\partial x_1 \partial x_2 \cdots x_p} f(\bold{x}) d\bold{x} \leq  B_p \left\{ (D\sqrt{\pi})^p \prod_{i = \lfloor p/2 \rfloor}^{p} i \right\} .
    \]
    Combining all these terms together produces 
    \[
        V_{HK}(L_n(\theta)) & \leq (2\gamma_n^\prime)^p \left\{ \frac{D \sqrt{\eta_2p^3\log(n)} }{n^{p(p-1)-1/2} } + B_{p} \frac{D^{(p-1)} (\eta_2\pi p)^{(p+2)/2} }{2^{(p-1)/2}n}+ B_p (D\sqrt{\pi})^p \prod_{i = \lfloor p/2 \rfloor}^{p} i \right\},
    \]
    which gives the desired bound.
    If $p$ is fixed, then this function is $O(1)$.
    As for the order of this function as $p\rightarrow \infty$, noting that $B_{p} < (p/\log(p+1))^p$ and that the falling factorial
    \*[
        \prod_{i = \lfloor p/2 \rfloor}^{p} i &\leq p^{\lfloor p/2 \rfloor}\exp\left(- \frac{\lfloor p/2 \rfloor(\lfloor p/2 \rfloor - 1)}{2p} \right)\\
        &\leq p^{ p/2 }\exp\left(- \frac{ (p/2 - 1) ( p/2  - 2)}{2p} \right)
    \]
    these facts imply
    \*[
       &B_{p } \frac{C^{(p-1)}  p^{(p+2)/2} }{n}+B_p (D\sqrt{\pi})^p \prod_{i = \lfloor p/2 \rfloor}^{p} i \\
        &\leq \frac{p^{p}}{n \log(p+1)^p}C^{(p-1)}  p^{(p+2)/2} +\frac{(D\sqrt{\pi})^p p^{ 3p/2 }}{\log(p+1)^p} \exp\left(- \frac{ (p/2 - 1) ( p/2  - 2)}{2p} \right)  \\
        &\leq \frac{(C')^{(p -1)/2} p^{ 3p/2 +1 }}{\log(p+1)^p} \left\{\frac{1}{n} + p^{-1}\exp\left(- \frac{ (p/2 - 1) ( p/2  - 2)}{2p} \right) \right\}\\
 &=O\left( \frac{(C')^{(p -1)/2} p^{ (3p + 2)/2 }}{n\log(p+1)^p} \right),
    \]
    for some $C^\prime > 0$, the result now follows.
\end{proof}

\begin{lemma}\label{lemma:chisqured}
    The following holds:
    \*[
        &E_{X \sim N(1/2, 4\eta_2p\log(n)I_p)}\left[\lVert X - 1/2 \rVert_2^p \cdot I\{X \in ([0,1]^p)^C\}\right] \\
        &\leq \{\eta_2p\log(n)\}^{-p/2}\frac{\Gamma(p)}{\Gamma(p/2)},
    \]
\end{lemma}

\begin{proof}
    Note that $\lVert X - 1/2 \rVert_2^2 \sim \{4\eta_2p\log(n)\}^{-1} \chi_p^2$ and for $C \sim  \chi_p^2$:
    \*[
        E[\lVert X - 1/2 \rVert_2^p]  &=\left\{4\eta_2p\log(n)\right\}^{-p/2} E[C^{p/2}]  \\
        &= \{4\eta_2p\log(n)\}^{-p/2} \left( 2^{p/2} \frac{\Gamma(p)}{\Gamma(p/2)}\right)\\
        &=\{2\eta_2p\log(n)\}^{-p/2} \frac{\Gamma(p)}{\Gamma(p/2)}.
    \]
    Note that the truncated integral is necessarily smaller than the full integral as the integrand is positive, this shows the desired result.
\end{proof}

\begin{lemma}\label{lemma:exp_diff}
    Under Assumptions \ref{assn:hess} and \ref{assn:partial} for $|\alpha| < p$
    \*[
    \exp\{ l_n\{2\gamma_n(x_\alpha -1/2) + \hat\theta_{n,\alpha}, \gamma_n\dot1_{-\alpha}+ \hat\theta_{-\alpha}\} - l_n(\hat\theta_{n})\} \leq \exp\left\{ -\frac{\eta_2 t(n)(p - |\alpha|)}{\eta_1}  \right\} .
    \]
\end{lemma}

\begin{proof}
    By a second-order Taylor expansion:
    \*[
        &l_n\{2\gamma_n(x_\alpha -1/2) + \hat\theta_{n,\alpha}, \gamma_n\dot1_{-\alpha}+ \hat\theta_{-\alpha}\} - l_n(\hat\theta_{n}) \\
        &\leq(2\gamma_n(x_\alpha -1/2), \gamma_n \cdot 1_{-\alpha})^\top \frac{\partial^2 l_n\{\tilde{\theta} \}}{\partial\theta^\top \partial\theta} (2\gamma_n(x_\alpha -1/2), \gamma_n \cdot 1_{-\alpha})\\
        &\leq - \gamma_n^2 n \eta_2 \lVert (2(x_\alpha -1/2),  1_{-\alpha})\rVert_2^2 =- 2 t(n)\frac{\eta_2}{\eta_1}\left\lVert x_{\alpha} - \frac{1}{2} \right\rVert_2^2 -\frac{\eta_2 t(n)(p - |\alpha|)}{\eta_1 }, 
    \]
    for some $\tilde\theta$ whose components $\theta_i =  (1 - t_i)\hat\theta_{n,i} + t_i2\gamma_n(x_\alpha -1/2) $ for some $t_i \in [0,1]$ for all $i \in [p]$, ignoring the first terms gives the first result.
\end{proof}
\end{appendix}

\begin{acks}[Acknowledgments]
We would like to thank Randolf Altmeyer for reminding us of the AM-GM inequality used in the proof of Lemma \ref{lem:halton}.
\end{acks}

\bibliographystyle{imsart-nameyear}
\bibliography{biblio}

\begin{thebibliography}{12}

\bibitem[\protect\citeauthoryear{Atanassov}{2004}]{atanassov2004discrepancy}
\begin{barticle}[author]
\bauthor{\bsnm{Atanassov},~\bfnm{Emanouil~I}\binits{E.~I.}}
(\byear{2004}).
\btitle{On the discrepancy of the Halton sequences}.
\bjournal{Math. Balkanica (NS)}
\bvolume{18}
\bpages{15--32}.
\end{barticle}
\endbibitem

\bibitem[\protect\citeauthoryear{Del~Moral, Doucet and
  Jasra}{2006}]{del2006sequential}
\begin{barticle}[author]
\bauthor{\bsnm{Del~Moral},~\bfnm{Pierre}\binits{P.}},
  \bauthor{\bsnm{Doucet},~\bfnm{Arnaud}\binits{A.}} \AND
  \bauthor{\bsnm{Jasra},~\bfnm{Ajay}\binits{A.}}
(\byear{2006}).
\btitle{Sequential monte carlo samplers}.
\bjournal{J. R. Stat. Soc. Ser. B Stat. Methodol.}
\bvolume{\bf{68}}
\bpages{411--436}.
\bdoi{10.1111/j.1467-9868.2006.00553.x}
\end{barticle}
\endbibitem

\bibitem[\protect\citeauthoryear{Fan and Lv}{2008}]{Fan}
\begin{barticle}[author]
\bauthor{\bsnm{Fan},~\bfnm{Jianqing}\binits{J.}} \AND
  \bauthor{\bsnm{Lv},~\bfnm{Jinchi}\binits{J.}}
(\byear{2008}).
\btitle{Sure independence screening for ultrahigh dimensional feature space}.
\bjournal{J. R. Stat. Soc. Ser. B Stat. Methodol.}
\bvolume{\bf{70}}
\bpages{849--911}.
\bdoi{https://doi.org/10.1111/j.1467-9868.2008.00674.x}
\end{barticle}
\endbibitem

\bibitem[\protect\citeauthoryear{Jiang, Wand and
  Bhaskaran}{2022}]{jiang2022usable}
\begin{barticle}[author]
\bauthor{\bsnm{Jiang},~\bfnm{Jiming}\binits{J.}},
  \bauthor{\bsnm{Wand},~\bfnm{Matt~P}\binits{M.~P.}} \AND
  \bauthor{\bsnm{Bhaskaran},~\bfnm{Aishwarya}\binits{A.}}
(\byear{2022}).
\btitle{Usable and precise asymptotics for generalized linear mixed model
  analysis and design}.
\bjournal{J. R. Stat. Soc. Ser. B Stat. Methodol.}
\bvolume{\bf{84}}
\bpages{55--82}.
\bdoi{https://doi.org/10.1111/rssb.12473}
\end{barticle}
\endbibitem

\bibitem[\protect\citeauthoryear{Jiang, Wand and
  Ghosh}{2024}]{jiang2024preciseasymptoticslinearmixed}
\begin{barticle}[author]
\bauthor{\bsnm{Jiang},~\bfnm{Jiming}\binits{J.}},
  \bauthor{\bsnm{Wand},~\bfnm{Matt~P.}\binits{M.~P.}} \AND
  \bauthor{\bsnm{Ghosh},~\bfnm{Swarnadip}\binits{S.}}
(\byear{2024}).
\btitle{Precise Asymptotics for Linear Mixed Models with Crossed Random
  Effects}.
\bjournal{arXiv:2409.05066}.
\end{barticle}
\endbibitem

\bibitem[\protect\citeauthoryear{Lemieux}{2009}]{carlo2009quasi}
\begin{bbook}[author]
\bauthor{\bsnm{Lemieux},~\bfnm{C.}\binits{C.}}
(\byear{2009}).
\btitle{Quasi-Monte Carlo Sampling}.
\bpublisher{Springer Verlag, Berlin}.
\end{bbook}
\endbibitem

\bibitem[\protect\citeauthoryear{Owen}{2019}]{owen2019monte}
\begin{bmisc}[author]
\bauthor{\bsnm{Owen},~\bfnm{Art~B}\binits{A.~B.}}
(\byear{2019}).
\btitle{Monte carlo book: the quasi-monte carlo parts}.
\end{bmisc}
\endbibitem

\bibitem[\protect\citeauthoryear{Rosser}{1941}]{rosser1941explicit}
\begin{barticle}[author]
\bauthor{\bsnm{Rosser},~\bfnm{Barkley}\binits{B.}}
(\byear{1941}).
\btitle{Explicit bounds for some functions of prime numbers}.
\bjournal{Am. J. Math.}
\bvolume{\bf{63}}
\bpages{211--232}.
\bdoi{https://doi.org/10.2307/2371291}
\end{barticle}
\endbibitem

\bibitem[\protect\citeauthoryear{Stringer, Bilodeau and
  Tang}{2025}]{stringer2022fitting}
\begin{barticle}[author]
\bauthor{\bsnm{Stringer},~\bfnm{Alex}\binits{A.}},
  \bauthor{\bsnm{Bilodeau},~\bfnm{Blair}\binits{B.}} \AND
  \bauthor{\bsnm{Tang},~\bfnm{Yanbo}\binits{Y.}}
(\byear{2025}).
\btitle{Asymptotics of numerical integration for two-level mixed models}.
\bjournal{Bernoulli (accepted)}.
\end{barticle}
\endbibitem

\bibitem[\protect\citeauthoryear{Tang}{2024}]{TangMC}
\begin{barticle}[author]
\bauthor{\bsnm{Tang},~\bfnm{Yanbo}\binits{Y.}}
(\byear{2024}).
\btitle{A Note on Monte Carlo Integration in High Dimensions}.
\bjournal{Am. Stat.}
\bvolume{78}
\bpages{290--296}.
\bdoi{10.1080/00031305.2023.2267637}
\end{barticle}
\endbibitem

\bibitem[\protect\citeauthoryear{Tang and Reid}{2025}]{tang2021laplace}
\begin{barticle}[author]
\bauthor{\bsnm{Tang},~\bfnm{Yanbo}\binits{Y.}} \AND
  \bauthor{\bsnm{Reid},~\bfnm{Nancy}\binits{N.}}
(\byear{2025}).
\btitle{Laplace and saddlepoint approximations in high dimensions}.
\bjournal{Bernoulli}
\bvolume{\bf{31}}
\bpages{1759--1788}.
\bdoi{10.3150/24-BEJ1758}
\end{barticle}
\endbibitem

\bibitem[\protect\citeauthoryear{Wainwright}{2019}]{Wainwright_2019}
\begin{bbook}[author]
\bauthor{\bsnm{Wainwright},~\bfnm{Martin~J.}\binits{M.~J.}}
(\byear{2019}).
\btitle{High-Dimensional Statistics: A Non-Asymptotic Viewpoint}.
\bpublisher{Cambridge University Press, Cambridge}.
\end{bbook}
\endbibitem

\end{thebibliography}
\end{document}